\documentclass[12pt,reqno]{amsart}

\usepackage{fullpage,amsfonts,amsmath,amssymb,amscd}

\input xypic
    \usepackage[all, knot]{xy}
    \xyoption{arc}

\theoremstyle{plain}
\newtheorem{thm}{Theorem}[section]
\newtheorem{lem}[thm]{Lemma}

\newtheorem{prop}[thm]{Proposition}

\newtheorem{thrm}[thm]{Theorem}

\theoremstyle{Definition}
\newtheorem{example}[thm]{Example}

\newtheorem{defn}[thm]{Definition}

\newtheorem{rem}[thm]{Remark}

\newcommand{\Spec}{\mathrm{Spec}}

\newcommand{\plim}{\mathop{\varinjlim}\limits}
\newcommand{\NCSpec}{NC\mathrm{Spec}}

\newcommand{\Rings}{\mathbf{Rings}}

\author{Richard Vale}
\address{Department of Mathematics, Cornell University, Malott Hall, Ithaca N.Y. 14853-4201, United States}

\title{On the opposite of the category of rings}
\begin{document}
\begin{abstract}
For every ring $R$, we construct a ringed space $\NCSpec(R)$ and for every ring homomorphism $R \rightarrow S$, a morphism of ringed spaces $\NCSpec(S) \rightarrow \NCSpec(R)$. We show that this gives a fully faithful contravariant functor from the category of rings to a category of ringed spaces. If $R$ is a commutative ring, we show that $\NCSpec(R)$ can be viewed as a natural completion of $\mathrm{Spec}(R)$. We then explain how the spaces $\NCSpec(R)$ may be glued, and study quasicoherent sheaves on them. As an example, we compute the category of quasicoherent sheaves on a space constructed from a skew-polynomial ring $R$ by an analogue of the $\mathrm{Proj}$ construction.
\end{abstract}
\maketitle
\section{Introduction}
A ringed space is a topological space $X$ equipped with a sheaf of rings. It is a well-known fact from algebraic geometry that the category of commutative rings can be identified with a category of ringed spaces. More precisely, there is a contravariant functor from the category of commutative rings to the category of ringed spaces, which associates to a ring $R$ its prime spectrum $\mathrm{Spec}(R)$ equipped with the structure sheaf $\mathcal{O}$. This functor is faithful. The spaces $(\mathrm{Spec}(R),\mathcal{O})$ are called affine schemes. The notion of affine scheme first appeared in \cite{EGA1} and other expositions may be found in \cite[Chapter V]{Shafo}, \cite{Hartshorne} and \cite{EisenbudHarris}. The aim of the present paper is to make a similar construction for rings which are not necessarily commutative.
\section{Conventions}\label{conventions}
Here we list some conventions which will be used in this paper.
\subsection*{}
We denote by $\Rings$ the category of all unital rings \emph{including the zero ring}. That is, all the rings $R$ which we consider have an identity element $1_R$ satisfying $x1_R=1_R x=x$ for all $x \in R$. The zero ring has only one element $0$, and its identity element is $1_0=0$.
We require ring homomorphisms to preserve the identity elements. Thus, the zero ring $0$ is a final object in the category $\Rings$, and if $R$ is not the zero ring, then there is no homomorphism $0 \rightarrow R$.
\begin{defn}
Let $R$ be a ring. An element $x \in R$ is a \emph{unit} if there exists $y \in R$ with $yx=xy=1_R$.
\end{defn}
Note that $0$ is a unit in the zero ring. Note also that if $\theta: R \rightarrow S$ is a ring homomorphism and $u \in R$ is a unit, then $\theta(u)$ is a unit in $S$.
\subsection{Ringed spaces}
As in algebraic geometry, we consider topological spaces with sheaves of rings on them.
\begin{defn}
If $X$ is a topological space then we denote by $\Omega(X)$ the collection of all open subsets of $X$. We regard $\Omega(X)$ as a poset ordered by inclusion, and hence also as a category.
\end{defn}
\begin{defn}
A \emph{sheaf of rings} on a topological space $X$ is a functor $\Omega(X)^{op} \rightarrow \Rings$ satisfying the sheaf axioms. (For the sheaf axioms, see \cite[V.2]{Shafo}.)
\end{defn}
\begin{defn}
A \emph{ringed space} is a pair $(X, \mathcal{O}_X)$ consisting of a topological space $X$ together with a sheaf of rings $\mathcal{O}_X$ on $X$.

A morphism $(X, \mathcal{O}_X) \rightarrow (Y, \mathcal{O}_Y)$ of ringed spaces is a pair $(\phi,\phi^\#)$ where $\phi : X \rightarrow Y$ is a continuous function and $\phi^\# : \mathcal{O}_Y \rightarrow \phi_* \mathcal{O}_X$ is a morphism of sheaves of rings on $Y$.

We denote the category of ringed spaces by $\mathbf{RingedSp}$.
\end{defn}
Notice that our definition of ringed space differs from the usual one in algebraic geometry, in which the rings are assumed to be commutative. We wish to allow noncommutative rings. We are now ready to state our aim more precisely.
\section{Aim}\label{aim}
We wish to construct a faithful functor $\Rings^{op} \rightarrow \mathbf{RingedSp}$. In other words, we want to associate to each ring $R$ a ringed space $\NCSpec(R)$ in a functorial way. There is one very easy solution to this problem, namely for each ring $R$, take a one-point space with the sheaf whose sections over the unique nonempty open set are $R$. However, such a construction would not be very interesting. We therefore also wish our construction to coincide with $\mathrm{Spec}(R)$ when $R$ happens to be a commutative ring. We are not able to achieve this, but we will show that when $R$ is commutative, $\mathrm{Spec}(R)$ embeds naturally in the complement of the generic point in $\NCSpec(R)$ as a dense subspace.

The structure of this paper is as follows. In Section \ref{localizationsn}, we describe various functorial properties of the localization of a ring $R$ at a subset $A \subset R$, which is defined via a universal property. We put the localizations at finite subsets together to obtain a partially ordered set $L(R)$ from a ring $R$. In Section \ref{Alexandrov}, we equip $L(R)$ with a topology and a sheaf of rings. In Section \ref{sobersn}, we recall the soberification construction from topology. The soberification of $L(R)$ inherits a sheaf of rings from $L(R)$, and this is the space which we call $\NCSpec(R)$. In Section \ref{functoralitysn}, we prove our main theorem, Theorem \ref{maintheorem2}, which says that $\NCSpec$ is a fully faithful functor from $\Rings^{op}$ to a subcategory of $\mathbf{RingedSp}$ which we describe. In Section \ref{commutativesn}, we explain how $\Spec(R)$ embeds naturally in $\NCSpec(R)$ when $R$ is commutative, and we show that $\NCSpec(R)$ may be viewed as a certain completion of $\Spec(R)$ in a category of based $T_0$--spaces. We then give some examples in Section \ref{examplessn}. In Section \ref{sheavesandglueing}, we explain how the spaces $\NCSpec(R)$ may be glued along Ore localizations and define quasicoherent sheaves on them. After some calculations, we show in Proposition \ref{qcoh=grmod/tors}, as an example, that the category of left quasicoherent sheaves on a space constructed from the skew-polynomial ring $R=\mathbb{C}_\lambda[x_1, \ldots, x_n]$ is equivalent to the category of graded left $R$--modules modulo torsion. Finally, in Section \ref{conclusion}, we discuss how our definition of $\NCSpec(R)$ may be related to earlier notions defined by other authors.
\subsection{Acknowledgements}
The author wishes to thank Sefi Ladkani for providing useful information on the Alexandrov topology and sheaves. The author wishes to thank Yuri Berest, Greg Muller, Michael Wemyss and Geordie Williamson for valuable comments.
\section{Localization}\label{localizationsn}
Our constructions use the technique of noncommutative localization of a ring. In this section, we collect the definition and basic properties of this construction.
\begin{defn}
Let $R$ be a ring and let $A$ be a subset of $R$. A \emph{localization} of $R$ at $A$ consists of a ring $loc(R,A)$ together with a ring homomorphism $\alpha_A: R \rightarrow loc(R,A)$ such that $\alpha_A(a)$ is a unit for all $a \in A$ and such that the following universal property holds.

If $\theta: R \rightarrow S$ is a ring homomorphism such that $\theta(a)$ is a unit for all $a \in A$, then there exists a unique ring homomorphism $\theta' : loc(R,A)\rightarrow S$ such that the following diagram commutes.
$$
\xymatrix{
R \ar[drr]_{\theta} \ar[rr]^{\alpha_A} && loc(R,A) \ar[d]^{\theta'}\\
&& S}$$
\end{defn}
For every ring $R$ and every subset $A$ of $R$, the ring $loc(R,A)$ exists, by \cite[Theorem 2.1]{Cohn}. It may be constructed by adjoining elements $\{t_a: a \in A\}$ to $R$, and quotienting the resulting ring by the relations $at_a=t_aa=1$ for all $a \in A$. The ring $loc(R,A)$ can of course be zero. For example, for any ring $R$, if $A=\{0\}$ then $loc(R,A)=0$. Or if $R=M_n(\mathbb{C})$ and $x \in R$ is a singular matrix then $loc(R,\{x\})=0$ (see Example \ref{matrix}). As a matter of notation, we will often write $loc(R,x)$ instead of $loc(R,\{x\})$ when we are localizing at a single element.

In general $loc(R,A)$ is diffcult to calculate explicitly. However, it has very nice functorial properties which we now list.
\begin{prop}\label{epi}
For any ring $R$ and any subset $E \subset R$, the ring homomorphism $\alpha_E :R \rightarrow loc(R,E)$ is epic in the category of rings (meaning that if $T$ is a ring and $f,g: loc(R,E) \rightarrow T$ are ring homomorphisms with $f\alpha_E=g\alpha_E$ then $f=g$).
\end{prop}
\begin{proof}
The proposition follows immediately from the uniqueness of the map $loc(R,E) \rightarrow T$ given by the universal property.
\end{proof}
\begin{defn}\label{preorderdef}
Let $R$ be a ring. If $A$ and $B$ are subsets of $R$ then we write $A \preceq B$ to mean that $\alpha_B(a)$ is a unit in $loc(R,B)$ for all $a \in A$.
\end{defn}
The relation $\preceq$ is well-defined because $loc(R,B)$ is unique up to canonical isomorphism, since it is defined via a universal property.
\begin{prop}\label{preorder}
Let $R$ be a ring. Suppose $A$ and $B$ are subsets of $R$ and $A \preceq B$. Then there is a unique ring homomorphism $p_{BA}:loc(R,A)\rightarrow loc(R,B)$ making the following diagram commute.
$$\xymatrix{
loc(R,A) \ar[rr]^{p_{BA}}&& loc(R,B)\\
& R \ar[ul]^{\alpha_A} \ar[ur]_{\alpha_B}&
}$$
For all $A$, $p_{AA}$ is the identity $loc(R,A) \rightarrow loc(R,A)$.
If $A \preceq B$ and $B \preceq C$ then $A \preceq C$ and $p_{CA}=p_{CB}p_{BA}$.
\end{prop}
\begin{proof}
If $A \preceq B$ then $p_{BA}$ exists by the universal property of $loc(R,B)$. Since the identity homomorphism satisfies the definition of $p_{AA}$, we get that $p_{AA}$ must be the identity.

If $A \preceq B$ and $B \preceq C$ then the two triangles in the following diagram commute.
$$\xymatrix{
loc(R,A) \ar[rr]^{p_{BA}}&& loc(R,B)\ar[rr]^{p_{CB}} && loc(R,C)\\
   &&    R\ar[ull]^{\alpha_A}\ar[u]^{\alpha_B} \ar[urr]_{\alpha_C}
                 &&
}$$
If $a \in A$ then $\alpha_C(a)=p_{CB}\alpha_B(a)=p_{CB}p_{BA}\alpha_A(a)$ which is a unit in $loc(R,C)$ since ring homomorphisms take units to units. Therefore $A \preceq C$ and there is a unique ring homomorphism $p_{CA} : loc(R,A) \rightarrow loc(R,C)$ satisfying $\alpha_C=p_{CA}\alpha_A$. By uniqueness, $p_{CA}=p_{CB}p_{BA}$.
\end{proof}
\begin{prop}\label{thetaa}
Let $\theta: R \rightarrow S$ be a ring homomorphism and let $A$ be a subset of $R$. Let $\alpha_{\theta(A)}: S \rightarrow loc(S,\theta(A))$ be the localization of $S$ at the subset $\theta(A)$. Then there is a unique ring homomorphism $\theta_A: loc(R,A) \rightarrow loc(S, \theta(A))$ such that the following diagram commutes.
$$
\xymatrix{
R \ar[d]_{\alpha_A}\ar[rr]^\theta && S\ar[d]^{\alpha_{\theta(A)}}\\
loc(R,A) \ar[rr]^{\theta_A}  && loc(S,\theta(A))
}$$
\end{prop}
\begin{proof}
For all $a\in A$, $\alpha_{\theta(A)}(\theta(a))=(\alpha_{\theta(A)}\theta)(a)$ is a unit in $loc(S,\theta(A))$. Therefore, a unique such $\theta_A$ exists by the universal property of $loc(R,A)$.
\end{proof}
\begin{prop}\label{theta}
Let $R$ be a ring and let $A \preceq B$ be subsets of $R$. Let $\theta: R \rightarrow S$ be a ring homomorphism. Then $\theta(A) \preceq \theta(B)$ and the following diagram commutes.
$$
\xymatrix{
loc(R,A) \ar[d]_{p_{BA}} \ar[rr]^{\theta_A}&& loc(S,\theta(A))\ar[d]^{p_{\theta(B)\theta(A)}}\\
loc(R,B) \ar[rr]^{\theta_B}&&loc(S,\theta(B))
}$$
Furthermore, the above diagram is a pushout in the category of rings.
\end{prop}
\begin{proof}
Consider the following commutative diagram.
$$
\xymatrix{
loc(R,A) \ar[dd]_{p_{BA}} \ar[rrr]^{\theta_A} &&& loc(S,\theta(A))
\\
  & R \ar[ul]^{\alpha_A}\ar[dl]_{\alpha_B} \ar[r]^\theta & S \ar[ur]_{\alpha_{\theta(A)}} \ar[dr]_{\alpha_{\theta(B)}}&\\
  loc(R,B) \ar[rrr]^{\theta_B} &&& loc(S,\theta(B))
  }
  $$
If $a \in A$, then $\alpha_{\theta(B)}(\theta(a))=\theta_Bp_{BA}\alpha_A(a)$ is a unit, so $\theta(A) \preceq \theta(B)$ and there exists a unique $p_{\theta(B)\theta(A)}$ making the right hand triangle in the diagram commute.
$$
\xymatrix{
loc(R,A) \ar[dd]_{p_{BA}} \ar[rrr]^{\theta_A} &&& loc(S,\theta(A))
\ar[dd]^{p_{\theta(B)\theta(A)}}
\\
  & R \ar[ul]^{\alpha_A}\ar[dl]_{\alpha_B} \ar[r]^\theta & S \ar[ur]_{\alpha_{\theta(A)}} \ar[dr]_{\alpha_{\theta(B)}}&\\
  loc(R,B) \ar[rrr]^{\theta_B} &&& loc(S,\theta(B))
  }
  $$
We wish to show that the outer square commutes. But $\alpha_{\theta(B)}(\theta(a))$ is a unit for all $a \in A$, and therefore there exists a unique $\gamma: loc(R,A) \rightarrow loc(S, \theta(B))$ with $\alpha_{\theta(B)}\theta =\gamma\alpha_A$. Commutativity of the triangles and the two small squares implies that both $\theta_B p_{BA}$ and $p_{\theta(B)\theta(A)}\theta_A$ satisfy the defining property of $\gamma$ and therefore $\theta_B p_{BA} =p_{\theta(B)\theta(A)}\theta_A$, as required.

Now we wish to show that the given diagram is a pushout. Suppose we are given a ring $T$ and ring homomorphisms $\lambda: loc(S,\theta(A))\rightarrow T$ and $\mu : loc(R,B) \rightarrow T$ such that $\mu p_{BA}=\lambda \theta_A$. Then consider the diagram
$$
\xymatrix{
loc(R,A) \ar[dd]_{p_{BA}} \ar[rrr]^{\theta_A} &&& loc(S,\theta(A))
\ar[dd]_{p_{\theta(B)\theta(A)}} \ar[dddr]^\lambda &
\\
  & R \ar[ul]^{\alpha_A}\ar[dl]_{\alpha_B} \ar[r]^\theta & S \ar[ur]^{\alpha_{\theta(A)}} \ar[dr]_{\alpha_{\theta(B)}}&&\\
  loc(R,B) \ar[rrr]^{\theta_B} \ar[rrrrd]_{\mu} &&& loc(S,\theta(B))&\\
  &&&& T}
  $$
If $b \in B$ then $\lambda \alpha_{\theta(A)} \theta(b)=\lambda \theta_A \alpha_A(b)= \mu p_{BA} \alpha_A (b) = \mu \alpha_B(b)$ is a unit in $T$. So there exists a unique 
$\rho: loc(S,\theta(B)) \rightarrow T$ with $\rho \alpha_{\theta(B)}=\lambda \alpha_{\theta(A)}$.
Now, $\rho p_{\theta(B)\theta(A)} \alpha_{\theta(A)}=\rho \alpha_{\theta(B)} =\lambda \alpha_{\theta(A)}$ and so $\rho p_{\theta(B)\theta(A)}=\lambda$  by Proposition \ref{epi}. Also, $\rho \theta_B \alpha_B= \rho \alpha_{\theta(B)}\theta=\lambda \alpha_{\theta(A)}\theta =\lambda \theta_A \alpha_A =\mu p_{BA} \alpha_A =\mu \alpha_B$, and so $\rho \theta_B=\mu$ by Proposition \ref{epi}.

Therefore, the following diagram commutes.
$$
\xymatrix{
loc(R,A) \ar[d]_{p_{BA}} \ar[rr]^{\theta_A}&& loc(S,\theta(A))\ar[d]_{p_{\theta(B)\theta(A)}} \ar[ddr]^\lambda &\\
loc(R,B) \ar[rrrd]_\mu \ar[rr]^{\theta_B}&&loc(S,\theta(B))\ar[dr]^\rho &\\
&&& T
}$$
It remains to show that $\rho$ is unique. If $\rho' : loc(S,\theta(B)) \rightarrow T$ is another ring homomorphism making the diagram commute, then $\rho' \alpha_{\theta(B)} = \rho' p_{\theta(B)\theta(A)} \alpha_{\theta(A)} =\lambda \alpha_{\theta(A)} =\rho \alpha_{\theta(B)}$. Therefore $\rho=\rho'$ by Proposition \ref{epi}.
\end{proof}
Now we explain how to put the various localizations $loc(R,A)$ together to make an ordered set. Recall that a \emph{preorder} on a set is a reflexive, transitive binary relation.
\begin{defn}
Let $R$ be a ring. We denote by $L_0(R)$ the set of all \emph{finite} subsets of $R$.
\end{defn}
\begin{lem}
The relation $A \preceq B$ of Definition \ref{preorderdef} is a preorder on $L_0(R)$.
\end{lem}
\begin{proof}
Follows from Proposition \ref{preorder}.
\end{proof}
It is a standard fact (which is easy to check) that any preorder $\preceq$ on a set $X$ defines an equivalence relation $\sim$ via $a \sim b$ if and only if $a \preceq b$ and $b \preceq a$, and that the set of equivalence classes $\{[a]: a \in X\}$ becomes a partially ordered set (poset) under the ordering $[a] \le [b]$ if and only if $a \preceq b$.
\begin{defn}
Let $R$ be a ring. We denote by $(L(R), \le)$ the poset of equivalence classes $L_0(R)/\sim$, with the partial order $\le$ induced by the preorder $\preceq$.

We will usually denote the equivalence class of $A \in L_0(R)$ by $R_A$.
\end{defn}
Although the symbol $R_A$ does not denote a ring, the notation is supposed to be reminiscent of the notation often used for localizations of commutative rings. So for example if $R$ is commutative and $f \in R$ then $R_{\{f\}}=R_{\{f^2\}}$ etc.

Before proceeding, we note here one useful property of the ordering on $L(R)$.
\begin{defn}\label{jsem}
If $(X, \le)$ is a poset and $a,b \in X$ then a \emph{join} of $a$ and $b$ is a least upper bound for $a$ and $b$. The poset $X$ is called a \emph{join semilattice} if every pair of elements of $X$ have a join.
\end{defn}
\begin{prop}\label{join}
Let $R$ be a ring. Then $(L(R),\le)$ is a join semilattice.
\end{prop}
\begin{proof}
Let $R_A$ and $R_B$ be two equivalence classes in $L(R)$ with equivalence class representatives $A, B \in L_0(R)$ respectively. Then $A \cup B \in L_0(R)$, and $\alpha_{A\cup B}(a)$ and $\alpha_{A \cup B}(b)$ are units in $loc(R, A \cup B)$ for all $a \in A$ and all $b \in B$. Therefore, $R_A \le R_{A \cup B}$ and $R_B \le R_{A \cup B}$. Furthermore, if $R_A \le R_C$ and $R_B \le R_C$ then $\alpha_C(a)$ and $\alpha_C(b)$ are units in $loc(R,C)$ for all $a \in A$ and all $b \in B$. Therefore, $\alpha_C(x)$ is a unit in $loc(R,C)$ for all $x \in A \cup B$ and hence $R_{A \cup B} \le R_C$. So $R_{A \cup B}$ is the join of $R_A$ and $R_B$.
\end{proof}
Notice that the proof of Proposition \ref{join} shows that $R_{A \cup B}$ is well-defined, independent of the choice of equivalence class representatives $A$ and $B$.

We refer to $L(R)$ with its ordering $\le$ as the \emph{localization semilattice} of $R$. We now explain how to make $L(R)$ into a ringed space.
\section{The Alexandrov topology}\label{Alexandrov}
We make $L(R)$ into a ringed space by giving it the most obvious topology and then defining a sheaf of rings on this topology in a tautological way. We caution the reader that the space constructed in this section is \emph{not} $\NCSpec(R)$. The ringed space we are about to define is, however, a necessary stepping-stone in the construction of $\NCSpec(R)$.
\begin{defn}
Let $(X, \le)$ be a poset. A subset $U$ of $X$ is called an \emph{upper set} if for all $a,b \in X$, if $a \in U$ and $a \le b$ then $b \in U$. The \emph{Alexandrov topology} on $X$ is the topology whose open sets are the upper sets.
\end{defn}
If $X$ is a poset and $a \in X$ then we denote by $U_a$ the set $\{b \in X: b \ge a\}$. It is easy to check that the Alexandrov topology really is a topology on $X$, and that the $U_a$ form a base of open sets of $X$. The Alexandrov topology is always $T_0$ (meaning that for any two points $a, b \in X$, there exists an open set containing one of $a,b$ but not the other). If $(X, \le_X)$ and $(Y,\le_Y)$ are posets, then a function $f : X \rightarrow Y$ is continuous for the Alexandrov topology if and only if it is order-preserving (meaning by definition that for all $x,y \in X$, $x \le_X y$ implies $f(x) \le_Y f(y)$).

We note here a property of basic open sets in the Alexandrov topology which will be useful in the sequel.
\begin{defn}\label{completeirred}
If $X$ is a topological space then $X$ is called \emph{completely $\cup$--irreducible} if whenever $\{U_\lambda\}_{\lambda \in \Lambda}$ is an open cover of $X$, we have $U_\lambda=X$ for some $\lambda$.
\end{defn}
\begin{prop}\label{basic}
If $(X,\le)$ is a poset with the Alexandrov topology then an open subset $U$ of $X$ is completely $\cup$--irreducible if and only if $U=U_x$ for some $x \in X$.
\end{prop}
Proposition \ref{basic} is trivial to prove because if $U_x = \bigcup U_\lambda$ then $x$ must belong to $U_\lambda$ for some $\lambda$ and so $U_x \subset U_\lambda$. It is useful to us because it expresses the property that an open set belongs to the canonical basis in purely topological terms.

Henceforth we consider $L(R)$ with the Alexandrov topology. For $R_E \in L(R)$, we write $U_E$ for the basic open set $\{R_F: R_F \ge R_E\}$. We now wish to construct a sheaf of rings on $L(R)$.

In fact, for any poset $X$, 
sheaves of rings on $X$ correspond to functors $X \rightarrow \Rings$, where $X$ is regarded as a category with a single arrow $x \rightarrow y$ if and only if $x \le y$. See the proof of an analogous fact in \cite[Sections 2.2, 2.3]{Ladkani}. The definition of $L(R)$ is chosen so that we may define such a functor $L(R) \rightarrow \Rings$ by $R_E \mapsto loc(R,E)$.

In more detail, for each finite subset $E$ of $R$, we fix a localization $\alpha_E: R\rightarrow loc(R,E)$. We define our sheaf $\mathcal{O}$ on basic open sets by $\mathcal{O}(U_E)=loc(R,E)$. To prove that this gives a sheaf, we use the following lemma from algebraic geometry. The lemma and its proof are taken from lecture notes by Vakil \cite[Class 5, Theorem 2.1]{vakil}. It is stated there for sheaves of sets, but the proof works equally well for sheaves of rings.
\begin{lem}\label{vakil1}
Let $X$ be a topological space. Let $\mathcal{B}$ be a base of open sets of $X$ and suppose for each $B \in \mathcal{B}$, we are given a ring $O(B)$, and for each inclusion $B' \subset B$ of basic open sets, a ring homomorphism $res_{B'}^B: O(B) \rightarrow O(B')$ such that $res^B_B$ is the identity for all $B$, and if $B'' \subset B' \subset B$ then $res^B_{B''}=res^{B'}_{B''}res^B_{B'}$. Suppose that the following property holds:

$(*)$ If $B \in \mathcal{B}$ and $B=\bigcup_\alpha B_\alpha$ with $B_\alpha \in \mathcal{B}$ for all $\alpha$, and if we are given $s_\alpha \in O(B_\alpha)$ such that for all $\alpha, \beta$ and all $\mathcal{B} \owns B'' \subset B_\alpha \cap B_\beta$, we have $res^{B_\alpha}_{B''}s_\alpha = res^{B_\beta}_{B''} s_\beta$, then there exists a unique $s \in O(B)$ with $res^B_{B_\alpha}s =s_\alpha$ for all $\alpha$.

Then there is a unique sheaf of rings $\mathcal{O}$ on $X$ extending $O$.
\end{lem}
\begin{proof}
We sketch the proof, which works because giving a sheaf on a base is enough information to determine the stalks, and any sheaf may be reconstructed from its stalks.

More precisely, for $x \in X$ we define the stalk $\mathcal{O}_x =\plim_{x \in B \in \mathcal{B}}O(B)$. For $s \in O(B)$ and $x \in B$ we write $s(x)$ for the germ of $s$ in $\mathcal{O}_x$. For an open $U \subset X$, we define
\begin{multline*}
\mathcal{O}(U)=\{ (s_x) \in \prod_{x \in U} \mathcal{O}_x : \forall y \in U \text{ there is a } B_y \in \mathcal{B} \text{ with }\\ y \in B_y \subset U \text{ and an } s \in O(B_y) \text{ such that for all } z \in B_y, s(z)=s_z\}.
\end{multline*}
It is then straightforward to check that this $\mathcal{O}$ is a sheaf and that for $B \in \mathcal{B}$, $\mathcal{O}(B)$ is canonically isomorphic to $O(B)$.
\end{proof}
For the space $L(R)$ with the Alexandrov topology, we set $O(U_F)=loc(R,F)$. If $U_F \subset U_G$ then $R_F \ge R_G$ so we have the map $p_{FG} : loc(R,G) \rightarrow loc(R,F)$ from Proposition \ref{preorder}. These maps satisfy $p_{EE}=1$ for all $E$ and $p_{GE}=p_{GF}p_{FE}$ for $R_E \le R_F \le R_G$.
In order to apply Lemma \ref{vakil1} it remains to check the condition $(*)$ of Lemma \ref{vakil1}. But if $U_F$ is a basic open subset of $L(R)$ then $U_F$ is completely $\cup$--irreducible, so $(*)$ holds trivially in this case.

We have now constructed a ringed space $(L(R), \mathcal{O})$. If we wish this space to be related to $\mathrm{Spec}(R)$ when $R$ is commutative, it should have a property called \emph{sobriety} (see Definition \ref{sobriety} below). There is a method for making a space into a sober space called soberification. The space $\NCSpec(R)$ will be constructed from $L(R)$ via soberification. We now describe how to do this.
\section{Soberification}\label{sobersn}
The details of this construction are taken from the book \cite{Stone}. Recall that a topological space $X$ is said to be irreducible if $X$ is nonempty and every two nonempty open subsets of $X$ intersect.
\begin{defn}
Let $X$ be a topological space. Let $C$ be a closed subset of $X$ and let $x \in X$. Then $x$ is said to be a \emph{generic point} of $C$ if $\overline{\{x\}}=C$.
\end{defn}
\begin{defn}\label{sobriety}
Let $X$ be a topological space. Then $X$ is \emph{sober} if every irreducible closed subset of $X$ has a unique generic point.
\end{defn}
Every Hausdorff space is sober. A typical example of a non-Hausdorff sober space is the prime spectrum of a commutative ring $R$.
\begin{defn}
Let $X$ be a topological space. Let $S(X)$ be the set of all irreducible closed subsets of $X$. For $U$ an open subset of $X$, define $\widetilde{U} \subset S(X)$ by
$$\widetilde{U}=\{C \in S(X): C \cap U \neq \varnothing\}.$$
Then $\{\widetilde{U}: U \text{ open in } X\}$ is a topology on $S(X)$ and the space $S(X)$ equipped with this topology is called the \emph{soberification} of $X$.
\end{defn}
The following proposition justifies the name of soberification.
\begin{prop}\label{soberification}
For any topological space $X$, $S(X)$ is a sober space. The map
\begin{align*}
q: X &\rightarrow S(X)\\
 x &\mapsto \overline{\{x\}}
 \end{align*}
 is continuous, 
 and $q^{-1}(\widetilde{U})=U$ for all open $U \subset X$.

 There is an order-preserving bijection $\Omega(X) \rightarrow \Omega(S(X))$ given by $U \mapsto \widetilde{U}$, and this satisfies $\widetilde{U \cap V}=\widetilde{U} \cap \widetilde{V}$ for any open sets $U, V \subset X$ and $\widetilde{\bigcup U_i} = \bigcup \widetilde{U_i}$ for any collection $\{U_i\}$ of open subsets of $X$.

 The soberification $X \mapsto S(X)$ defines a functor from the category of topological spaces to itself.
\end{prop}
\begin{proof}
See \cite[II.1.6, II.1.7]{Stone}.
\end{proof}
For a ring $R$, we have already defined a sheaf of rings $\mathcal{O}$ on $L(R)$. Since there is an order-preserving bijection $U \mapsto \widetilde{U}$ between the topologies of $L(R)$ and $S(L(R))$, such that $\widetilde{U_1 \cap U_2}=\widetilde{U_1}\cap \widetilde{U_2}$ and $\widetilde{\bigcup_i U_i}=\bigcup_i \widetilde{U_i}$ for all $U_i$, we can then define a sheaf of rings on $S(L(R))$, which we also denote by $\mathcal{O}$.
\begin{defn}
Let $R$ be a ring. We denote by $\NCSpec(R)$ the ringed space $S(L(R))$ equipped with the sheaf of rings $\mathcal{O}$.
\end{defn}
Our next goal is to prove that $\NCSpec$ is a functor.
\section{Functorality}\label{functoralitysn}
We wish to show that the assignment $R \mapsto \NCSpec(R)$ defines a faithful contravariant functor $\Rings^{op} \rightarrow \mathbf{RingedSp}$. Let $R$ and $S$ be rings and let $\theta: R \rightarrow S$ be a ring homomorphism. We want to define a continuous map $S(L(S)) \rightarrow S(L(R))$. First, we define a function $L_0(R) \rightarrow L_0(S)$ by $E \mapsto \theta(E)$. By Proposition \ref{theta}, this map preserves the preorder $\preceq$ on $L_0(R)$. It therefore induces an order-preserving map $t_\theta: L(R) \rightarrow L(S)$, given by $t_\theta: R_E \mapsto S_{\theta(E)}$. By Proposition \ref{soberification}, this induces a continuous map $S(L(R)) \rightarrow S(L(S))$. However, this map goes in the wrong direction. It turns out that $t_\theta$ also induces a continuous map $S(L(S)) \rightarrow S(L(R))$. This happens because of a special property of the Alexandrov topology on a poset. We first require the following lemma (see \cite[Lemma 1.1, 1.2]{Hoffmann}).
\begin{lem}\label{irredclosed}
Let $(P, \le)$ be a poset. A subset $C$ of $P$ is closed for the Alexandrov topology if and only if $C$ is a \emph{lower} set. That is, for all $x,y \in P$, if $x \in C$ and $x \ge y$ then $y \in C$.

A subset $C$ of $P$ is irreducible if and only if $C$ is \emph{directed upwards}. That is, for all $x,y \in C$ there exists $z \in C$ with $x \le z$ and $y \le z$.
\end{lem}
\begin{proof}
The first statement is true because a subset of $P$ is a lower set if and only if it is the complement of an upper set.

For the second statement, suppose $C \subset P$ is irreducible. Then if $x,y \in C$, $U_x \cap U_y \cap C \neq \varnothing$. Therefore, there is $z \in C$ with $z \ge x,y$. Conversely, if $C$ is directed upwards, then for any two basic open sets $U_a, U_b$ with $U_a \cap C$, $U_b \cap C \neq\varnothing$, we get $U_a \cap U_b \cap C \neq \varnothing$. Therefore, any pair of nonempty open subsets of $C$ intersect, so $C$ is irreducible.
\end{proof}
\begin{lem}\label{backwards}
Let $(P, \le)$ and $(Q,\le)$ be join semilattices (see Definition \ref{jsem}), regarded as topological spaces equipped with the respective Alexandrov topologies. Let $f: (P, \le) \rightarrow (Q , \le)$ be an order-preserving map which preserves joins. Then for every irreducible closed subset $C$ of $Q$, $f^{-1}(C)$ is an irreducible closed subset of $P$, and $C \mapsto f^{-1}(C)$ is a continuous function $S(Q) \rightarrow S(P)$.
\end{lem}
\begin{proof}
Let $C$ be an irreducible closed subset of $Q$. Then $f^{-1}(C)$ is closed because order-preserving functions are continuous for the Alexandrov topology. So we only need to check that $f^{-1}(C)$ is irreducible. Let $x,y \in f^{-1}(C)$. Then $f(x), f(y) \in C$. So, since $C$ is directed upwards by Lemma \ref{irredclosed}, there exists $z \in C$ with $f(x) \le z, f(y) \le z$. Since $Q$ is a join semilattice, the join $f(x) \vee f(y) \le z$. So $f(x) \vee f(y) \in C$ since $C$ is a lower set. Since $f$ preserves joins by hypothesis, $f(x) \vee f(y)=f(x \vee y)$. Therefore, $x \vee y \in f^{-1}(C)$ and therefore $f^{-1}(C)$ is directed upwards. So $f^{-1}(C)$ is closed and irreducible, as required.

For the statement about continuity, let us temporarily write ${\hat{f}}$ for the function $S(Q) \rightarrow S(P)$ defined by ${\hat{f}}(C)=f^{-1}(C)$. We show that if $\widetilde{U_a}$ is a basic open subset of $S(P)$, then $({\hat{f}})^{-1}(\widetilde{U_a})=\widetilde{U_{f(a)}}$, which is open in $S(Q)$. We have
\begin{align*}
({\hat{f}})^{-1}(\widetilde{U_a})&=\{C \in S(Q): f^{-1}(C) \cap U_a
\neq \varnothing\}\\
&= \{C \in S(Q) : \exists x \in P \text{ with } f(x) \in C \text{ and } x \ge a\}.
\end{align*}
We claim that this set equals $\{ C \in S(Q) : C \cap U_{f(a)} \neq \varnothing\}$. Indeed, if $C \cap U_{f(a)} \neq \varnothing$ then there exists $z \in C$ with $z \ge f(a)$, so $f(a) \in C$ because $C$ is a lower set. On the other hand, if $x \ge a$ and $f(x) \in C$, then since $f$ is order-preserving, $f(x) \ge f(a)$ and so $f(a) \in C$ because $C$ is a lower set. So we have
$$
({\hat{f}})^{-1}(\widetilde{U_a})=\{ C \in S(Q) : C \cap U_{f(a)} \neq \varnothing\}=\widetilde{U_{f(a)}}$$
as required.
\end{proof}
Now let us return to the case of a ring homomorphism $\theta : R \rightarrow S$. We have seen that there is an order-preserving function $t_\theta: L(R) \rightarrow L(S)$ defined by $t_\theta: R_E \mapsto S_{\theta(E)}$. We check that $t_\theta$ preserves joins. By Proposition \ref{join}, the join of $R_E$ and $R_F$ is $R_{E \cup F}$. And $t_\theta(R_{E \cup F}) = S_{\theta(E \cup F)}=S_{\theta(E) \cup \theta(F)}$, so $t_\theta$ preserves joins.

Thus, by Lemma \ref{backwards}, $t_\theta$ induces a continuous function
$$\hat{\theta} : S(L(S)) \rightarrow S(L(R))$$
with $\hat{\theta}^{-1}(\widetilde{U_E})=\widetilde{U_{\theta(E)}}$ for all basic open sets $\widetilde{U_E}$ of $S(L(R))$.

The continuous function $\hat{\theta}$ is the first ingredient of our morphism of ringed spaces $\NCSpec(S) \rightarrow \NCSpec(R)$. It remains to define a morphism of sheaves $\mathcal{O}_{\NCSpec(R)} \rightarrow \hat{\theta}_* \mathcal{O}_{\NCSpec(S)}$. As in Lemma \ref{vakil1}, we construct this morphism first on a base, and then extend it uniquely to a morphism of sheaves. To do this, we use the following lemma from algebraic geometry. It is a special case of
\cite[Class 5, Exercise 2.C]{vakil}.
\begin{lem}\label{vakil2}
Let $X$ and $Y$ be topological spaces. Let $\mathcal{O}_X$ be a sheaf of rings on $X$ and let $\mathcal{O}_Y$ be a sheaf of rings on $Y$. Let $f : X \rightarrow Y$ be a continuous function. Suppose there are bases $\mathcal{B}_X$ and $\mathcal{B}_Y$ of $X$ and $Y$ respectively, such that $f^{-1}(B) \in \mathcal{B}_X$ for every $B \in \mathcal{B}_Y$. Suppose that for every $B \in \mathcal{B}_Y$, there is a morphism $f_B :\mathcal{O}_Y(B) \rightarrow \mathcal{O}_X(f^{-1}(B))$ such that for every inclusion of basic open sets $B' \subset B$ on $Y$, the diagram
$$\xymatrix{
\mathcal{O}_Y(B) \ar[d] \ar[rr]^{f_B}&& \mathcal{O}_X(f^{-1}(B)) \ar[d]\\
\mathcal{O}_Y(B') \ar[rr]^{f_{B'}}&& \mathcal{O}_X(f^{-1}(B'))
}$$
commutes, where the vertical arrows are the restriction maps.

Then there exists a unique map $\mathcal{O}_Y \rightarrow f_*\mathcal{O}_X$ of sheaves on $Y$ which agrees with $f_B$ on every $B \in \mathcal{B}_Y$.
\end{lem}
\begin{proof}
We define a map on stalks as follows: if $y \in B \subset Y$ with $B \in \mathcal{B}_Y$ and $a(y):=[a,B]$ is the germ of a section $a\in \mathcal{O}_Y(B)$ at $y$, then we define $f_y([a,B])=[f_B(a), B] \in (f_*\mathcal{O}_X)_{y}$. Putting these maps together for $y \in Y$ gives a map between the sheafifications of $\mathcal{O}_X$ and $f_*\mathcal{O}_Y$, because the sections of the sheafification of a presheaf $\mathcal{F}$ over an open set $U$ are
\begin{multline*}\mathcal{F}^+(U)=\{(s_x)\in \prod_{x \in U}\mathcal{F}_x: \exists \text{ open cover } U = \bigcup U_i
\text{ and } s_i \in \mathcal{F}(U_i)\\ \text{ with }
s_i(x) =s_x \text{ for all } x \in U_i\}.\end{multline*}
But a map between the sheafifications is the same thing as a map of sheaves, so we are done.
\end{proof}
We apply Lemma \ref{vakil2} with $f=\hat{\theta}: S(L(S))\rightarrow S(L(R))$. We take the base $\mathcal{B}$ of $S(L(R))$ to be $\{\widetilde{U_E}: E \in L_0(R)\}$. By Lemma \ref{backwards}, $\hat{\theta}^{-1}(\widetilde{U_E})=\widetilde{U_{\theta(E)}}$. For each $\widetilde{U_E} \subset S(L(R))$, we have a map
$$\theta_E: loc(R,E) \rightarrow loc(S,\theta(E))$$
which comes from Proposition \ref{thetaa}. By Proposition \ref{theta}, these maps commute with the restrictions. So by Lemma \ref{vakil2}, there is a unique morphism of sheaves $$\hat{\theta}^\#:\mathcal{O}_{\NCSpec(R)} \rightarrow \hat{\theta}_* \mathcal{O}_{\NCSpec(S)}$$ which agrees with the $\theta_E$ on basic open sets, as required.
\begin{thrm}\label{maintheorem}
Let $R$ and $S$ be rings and let $\theta: R \rightarrow S$ be a ring homomorphism. Then $\theta$ determines a morphism of ringed spaces
$$(\hat{\theta},\hat{\theta}^\#) : \NCSpec(S) \rightarrow \NCSpec(R).$$
Furthermore, $\theta \mapsto (\hat{\theta}, \hat{\theta}^\#)$ defines a faithful contravariant functor $\Rings^{op} \rightarrow \mathbf{RingedSp}$.
\end{thrm}
\begin{proof}
To check the functorality, suppose
$$
\xymatrix{
R \ar[r]^{\theta} &S\ar[r]^{\varphi} &T}$$
are ring homomorphisms. We obtain $t_{\theta} : L(R) \rightarrow L(S)$ defined by $t_\theta(R_E)=S_{\theta(E)}$ and $t_\varphi: L(S) \rightarrow L(T)$ defined by $t_\varphi(S_F) =T_{\varphi(F)}$. It is clear that $t_{\varphi \theta} =t_\varphi t_\theta$. Now, for an irreducible closed subset $C \subset L(T)$, $\widehat{\varphi \theta}(C)=t_{\varphi \theta}^{-1}(C)=t_\theta^{-1}t_\varphi^{-1}(C)=\hat{\theta}\hat{\varphi}
(C)$, so $\widehat{\varphi \theta} =\hat{\theta}\hat{\varphi}$.

To show that $\widehat{\varphi\theta}^\# = (\hat{\theta}_*\hat{\varphi}^\#) \hat{\theta}^\#$, it suffices to show that for each finite subset $E$ of $R$, we have $(\varphi \theta)_E = \varphi_{\theta(E)}\theta_E$. Recalling Proposition \ref{thetaa}, all the squares in the following diagram commute.
$$
\xymatrix{
R  \ar[rr]^\theta \ar[dd]^{\alpha_E} && S \ar[rr]^\varphi \ar[d]^{\alpha_{\theta(E)}} && T \ar[dd]^{\alpha_{\varphi(\theta(E))}} \\
&& loc(S, \theta(E)) \ar[rrd]_{\varphi_{\theta(E)}} &&\\
loc(R,E) \ar[urr]_{\theta_E} \ar[rrrr]_{(\varphi \theta)_E}&&&& loc(T, \varphi(\theta(E)))
}
$$
By the uniqueness of $(\varphi \theta)_E$, we obtain $(\varphi \theta)_E = \varphi_{\theta(E)}\theta_E$ as desired.

To show that this functor is faithful, we must show that if $\theta: R \rightarrow S$ then $(\hat{\theta},\hat{\theta}^\#)$ uniquely determines $\theta$. But if we take $E=\{1_R\} \in L_0(R)$, then $\theta(E)=\{1_S\}$ and $\theta=\theta_E : loc(R,1_R) \rightarrow loc(S, 1_S)$ is the map of global sections $$\mathcal{O}_{\NCSpec(R)}(\NCSpec(R)) \rightarrow \mathcal{O}_{\NCSpec(S)}(\NCSpec(S))$$
which can be recovered from $\hat{\theta}^\#$.
\end{proof}
There is now a natural question: which morphisms of ringed spaces $\NCSpec(S) \rightarrow \NCSpec(R)$ are induced by ring homomorphisms $R \rightarrow S$? We now answer this question. Recall from Definition \ref{completeirred} that an open subset $U$ of a topological space $X$ is \emph{completely $\cup$--irreducible} if whenever $U=\bigcup U_\lambda$ is an open cover of $U$, we have $U=U_\lambda$ for some $\lambda$.
Complete $\cup$--irreducibility is a very strong form of compactness; every open cover has a one-element subcover.
\begin{defn}
Let $(\phi, \phi^\#): (X, \mathcal{O}_X)\rightarrow (Y,\mathcal{O}_Y)$ be a morphism of ringed spaces. We say $(\phi,\phi^\#)$ is \emph{prim} if, for every completely $\cup$--irreducible open $U \subset Y$, the following two conditions hold.
\begin{enumerate}
\item $\phi^{-1}(U)$ is completely $\cup$--irreducible.
\item For every completely $\cup$--irreducible open $V \subset U$, the diagram
$$\xymatrix{
\mathcal{O}_Y(U) \ar[rr]\ar[d] && (\phi_*\mathcal{O}_X)(U) \ar[d]\\
\mathcal{O}_{Y}(V) \ar[rr] && (\phi_*\mathcal{O}_X)(V)
}$$
is a pushout in the category of rings.
\end{enumerate}
\end{defn}
Note that identities are prim, and compositions of prim morphisms are prim, so that ringed spaces and prim morphisms between them form a category.
\begin{lem}\label{prim1}
Let $\theta: R \rightarrow S$ be a ring homomorphism. Then the induced morphism $(\hat{\theta},\hat{\theta}^\#): \NCSpec(S) \rightarrow \NCSpec(R)$ is prim.
\end{lem}
\begin{proof}
Let $\theta: R \rightarrow S$ be a ring homomorphism. We first consider the continuous function $\hat{\theta}:S(L(S)) \rightarrow S(L(R))$. In view of Proposition \ref{basic}, the completely $\cup$--irreducible open subsets of $S(L(R))$ are precisely those of the form $\widetilde{U_E}$ for $E \in L_0(R)$. We have already seen that $\hat{\theta}^{-1}(\widetilde{U_E})=\widetilde{U_{\theta(E)}}$, which is a completely $\cup$--irreducible open subset of $S(L(S))$.

Now let $\widetilde{U_F} \subset \widetilde{U_E}$ be an inclusion of completely $\cup$--irreducible open subsets of $S(L(R))$. Then $U_F \subset U_E$ as subsets of $L(R)$. We must show that the diagram
$$
\xymatrix{
loc(R,E) \ar[d]_{p_{FE}} \ar[rr]^{\theta_E} &&  loc(S,\theta(E)) \ar[d]^{p_{\theta(F)\theta(E)}}\\
loc(R,F) \ar[rr]^{\theta_F} && loc(S, \theta(F))
}
$$
is a pushout. This was proved in Proposition \ref{theta}.
\end{proof}
\begin{thrm}\label{maintheorem2}
Let $R$ and $S$ be rings. A morphism of ringed spaces $\varphi: \NCSpec(S) \rightarrow \NCSpec(R)$ is induced by a ring homomorphism $R \rightarrow S$ if and only if $\varphi$ is prim. Thus $R \mapsto \NCSpec(R)$ is a fully faithful embedding of the category of rings into the category of ringed spaces and prim morphisms.
\end{thrm}
\begin{proof}
Suppose $\varphi = (\phi, \phi^\#) : \NCSpec(S) \rightarrow \NCSpec(R)$ is a prim morphism of ringed spaces. Write $X=\NCSpec(S)$ and $Y=\NCSpec(R)$. Let $\theta: R \rightarrow S$ be the map of global sections defined by $\phi^\#$. We show that $(\phi, \phi^\#)=(\hat{\theta}, \hat{\theta}^\#)$. Let $\widetilde{U_E}$ be a basic open subset of $\NCSpec(R)$. Then $\widetilde{U_E}$ is completely $\cup$--irreducible, so $\phi^{-1}(\widetilde{U_E})$ is also completely $\cup$--irreducible and therefore equals $\widetilde{U_F}$ for some $F \in L_0(S)$, by Proposition \ref{basic}. The diagram
$$\xymatrix{
\mathcal{O}_Y(Y) \ar[d]_{\alpha_E} \ar[rr]^\theta && \mathcal{O}_X(X)\ar[d]^{\alpha_F}\\
\mathcal{O}_Y(\widetilde{U_E}) \ar[rr]^{\phi^\#_{\widetilde{U_E}}} && \mathcal{O}_X(\phi^{-1}(\widetilde{U_E}))
}$$
is a pushout, since $\phi$ is prim and $Y=\widetilde{U_{1_R}}$ is completely $\cup$--irreducible. But
$$
\xymatrix{
R \ar[d]_{\alpha_E} \ar[rr]^\theta && S \ar[d]^{\alpha_{\theta(E)}}\\
loc(R,E) \ar[rr]^{\theta_E} &&loc(S,\theta(E))
}$$
is also a pushout, by Proposition \ref{theta}. Thus, $loc(S,\theta(E))$ and $\mathcal{O}_X(\phi^{-1}(\widetilde{U_E}))=loc(S,F)$ are canonically isomorphic, and so $S_{\theta(E)}=S_F$ and therefore $\widetilde{U_F}=\widetilde{U_{\theta(E)}}$. Now we show that $\phi=\hat{\theta}$. If $x \in \NCSpec(S)$ then $x$ is an irreducible closed subset of $L(S)$. We show that $\phi(x)=\hat{\theta}(x)$. By definition, $\hat{\theta}(x)=\{R_E \in L(R): S_{\theta(E)} \in x\}$. Now since $\phi(x)$ is an irreducible closed subset of $L(R)$, we have $R_E \in \phi(x)$ if and only if $\phi(x) \in \widetilde{U_E}$ if and only if $x \in \phi^{-1}(\widetilde{U_E})=\widetilde{U_{\theta(E)}}$. So $R_E \in \phi(x)$ if and only if $S_{\theta(E)} \in x$ and therefore $\phi(x)=\hat{\theta}(x)$ as required.

To show that $\hat{\theta}^\#=\phi^\#$, we need only observe that for each $A \in L_0(R)$, the map $\theta_A$ of Proposition \ref{thetaa} is uniquely determined by $\theta$, by Proposition \ref{epi}.
\end{proof}
\begin{rem}\begin{rm}
Before closing this section, we make a remark about Theorem \ref{maintheorem2}. For commutative rings $R$ and $S$, it is well-known (\cite[Proposition 2.3]{Hartshorne}) that a morphism $\phi:\Spec(S) \rightarrow \Spec(R)$ is induced by a ring homomorphism $R \rightarrow S$ if and only if it is a local morphism of locally ringed spaces, that is, if and only if for each $x \in \Spec(S)$, the map on stalks $\mathcal{O}_{\Spec(R),\phi(x)} \rightarrow \mathcal{O}_{\Spec(S), x}$ is a local homomorphism of local rings. This condition can clearly be checked locally, meaning that a morphism $\phi: X \rightarrow Y$ of locally ringed spaces is local if and only if there is an open cover $Y=\bigcup_\alpha U_\alpha$ such that $\phi|_{\phi^{-1}(U_\alpha)}:\phi^{-1}(U_\alpha)\rightarrow U_\alpha$ is local for each $U_\alpha$. If there is to be any hope of generalising algebraic geometry to the spaces $\NCSpec(R)$, we also require that primness be a local property. This is in fact the case.
\begin{prop}
Let $\phi: X \rightarrow Y$ be a morphism of ringed spaces. Then $\phi$ is prim if and only if there is an open cover $Y =\bigcup_\alpha U_\alpha$ such that $\phi|_{\phi^{-1}(U_\alpha)}:\phi^{-1}(U_\alpha)\rightarrow U_\alpha$ is prim for each $\alpha$.
\end{prop}
\begin{proof}
The proposition follows immediately from the observation that if $V\subset Y$ is a completely $\cup$--irreducible open set then $V =\bigcup_\alpha (V \cap U_\alpha)$ and therefore $V \subset U_\alpha$ for some $\alpha$.
\end{proof}
\end{rm}
\end{rem}
\section{The commutative case}\label{commutativesn}
In this section, we wish to compare $\NCSpec(R)$ with $\Spec(R)$ when $R$ is a commutative ring. First, we note some very basic properties of the space $\NCSpec(R)$ for an arbitrary ring $R$.

Let $R$ be a ring. Then for any finite subset $E$ of $R$, we have $\{1_R\} \preceq E$ and $E \preceq \{0\}$ in the notation of Definition \ref{preorderdef}. So the localization semilattice $L(R)$ has a minimum element $R_{1_R}$ and a maximum element $R_0$. Since open sets are upper sets, $R_0$ belongs to every nonempty open set. Thus, $L(R)$ is itself an irreducible space and therefore determines a point $\gamma \in S(L(R))$. This point $\gamma$ belongs to every nonempty open $\widetilde{U} \subset S(L(R))$, and therefore $S(L(R))$ is an irreducible space. We have proved the following proposition.
\begin{prop}
If $R$ is a ring then $S(L(R))$ is irreducible.
\end{prop}
We therefore cannot hope that $\Spec(R)$ will be isomorphic to $\NCSpec(R)$ for a commutative ring $R$, since $\Spec(R)$ need not be irreducible. We will, however, show that the two are related.

Our first goal is to identify the points of $\NCSpec(R)$ when $R$ is commutative, and to show that there is a continuous map $\Spec(R) \rightarrow \NCSpec(R)$ whose image is dense in $\NCSpec(R) \setminus\{\gamma\}$, where $\gamma$ denotes the generic point of $\NCSpec(R)$.

If $R$ is a commutative ring, then for any finite subset $E$ of $R$, $loc(R,E) \cong loc(R, \prod_{e \in E}e)$. This is because in a commutative ring, $xy$ is a unit if and only if $x$ and $y$ are units, and $loc(R,E)$ is commutative by construction. So every element of $L(R)$ is of the form $R_f$ for some $f \in R$.
In what follows, we denote the basic open subset $\widetilde{U_{R_f}}$ of $S(L(R))$ by $\widetilde{U_f}$ for short.

Furthermore, for $f \in R$, $loc(R,f)$ coincides with the usual notion of localization at the set of powers of $f$. This is clear because both types of localization are defined by the same universal property.

Next, recall that a subset $S$ of a commutative ring $R$ is called \emph{multiplicative} if $1_R \in S$ and for all $x,y \in S$, $xy \in S$. Also, $S$ is called \emph{saturated} if $xy \in S$ implies $x \in S$.

We now identify the points of $S(L(R))$, or equivalently the irreducible closed subsets of $L(R)$, when $R$ is commutative.
\begin{lem}
Let $R$ be a commutative ring and let $C$ be an irreducible closed subset of $L(R)$. Then there exists a saturated multiplicatively closed subset $S$ of $R$ such that $C=\{R_f: f \in S\}$.
\end{lem}
\begin{proof}
Let $C$ be an irreducible closed subset of $L(R)$. Recall from Lemma \ref{irredclosed} that $C$ is a lower set and is directed upwards.

Suppose $R_f \in C$. Then $\alpha_f(1)$ is a unit in $loc(R,f)$. So $R_1 \le R_f$ and therefore $R_1$ belongs to $C$ since $C$ is a lower set. Now suppose $R_f, R_g \in C$. Then since $C$ is directed upwards, there exists $R_h \in C$ such that $\alpha_h(f)$ and $\alpha_h(g)$ are units in $loc(R,h)$. But then $\alpha_h(fg)$ is a unit and so $R_{fg} \le R_h$. Since $C$ is a lower set, we have $R_{fg} \in C$. So $S=\{f: R_f \in C\}$ is a multiplicatively closed set. Finally, if $xy \in S$ then $R_{xy} \in C$. Since $\alpha_{xy}(x)$ is a unit in $loc(R,xy)$, we have $R_x \le R_{xy}$. So $R_x \in C$ since $C$ is a lower set. Therefore, $S$ is saturated, as required.
\end{proof}
We now use the fact that every saturated multiplicatively closed subset of a commutative ring is the complement of a union of prime ideals (\cite[1.1, Theorem 2]{Kaplansky}) to obtain the following proposition.
\begin{prop}\label{unionofprimes}
If $R$ is a commutative ring, there is a bijection between unions of prime ideals of $R$ and irreducible closed subsets of $L(R)$, given by
$$\bigcup_{\lambda \in \Lambda} P_\lambda \mapsto \{R_f : f \notin \bigcup_{\lambda \in \Lambda} P_\lambda\}.$$
\end{prop}
\begin{proof}
We have already shown that every irreducible closed subset of $L(R)$ has the given form. We must show that if $P=\bigcup P_\lambda$ and $Q =\bigcup Q_\mu$ are unions of prime ideals of $R$ such that $\{R_f: f \notin P\}=\{R_f: f \notin Q\}$ then $P=Q$. By symmetry, it suffices to show that $Q\subset P$. Suppose $g \in Q \setminus P$. Then $R_g \in \{R_f: f \notin P\}=\{R_f: f \notin Q\}$. So $R_g=R_h$ for some $h \notin Q$. But $R_h \le R_g$ implies that $\alpha_g(h)$ is a unit in $loc(R,g)$, and by the usual theory of localization for commutative rings, this means that $g \in rad(h)$. So $R_h=R_g$ if and only if $rad(h)=rad(g)$. Thus $h$ and $g$ belong to the same prime ideals. But $g \in Q$ implies $g \in Q_\mu$ for some $\mu$ and therefore $h \in Q_\mu$, which contradicts the assumption that $h \notin Q$.

Thus the given map is a bijection. It remains to show that if $S$ is a saturated multiplicatively closed subset of $R$ then $C=\{R_f:f \in S\}$ is always an irreducible closed subset of $L(R)$. This follows from Proposition \ref{irredclosed}, because $S$ being multiplicative implies that $C$ is directed upwards, while $S$ being saturated implies that $C$ is a lower set.
\end{proof}
Thus, for a commutative ring $R$, the points of $\NCSpec(R)$ may be identified with unions $\bigcup_\lambda P_\lambda$ of prime ideals of $R$, and the basic open set $\widetilde{{U}_{f}}$ corresponds to $\{\bigcup_\lambda P_\lambda : f \notin P_\lambda \text{ for all } \lambda\}$.

There is thus a one-to-one function (on the level of sets) $\phi: \mathrm{Spec}(R) \rightarrow \NCSpec(R)$ given by
$$\phi(P)=\{R_f: f \notin P\}.$$
Recall that a base for the topology on $\mathrm{Spec}(R)$ is given by the sets $D(f)=\{P: f \notin P\}$ for $f \in R$ (see \cite[V.1.3]{Shafo}).
\begin{lem}\label{phicts}
Let $R$ be a commutative ring. Then
$$\phi: \Spec(R) \rightarrow \NCSpec(R)$$
is continuous, and indeed $\phi^{-1}(\widetilde{U_g})=D(g)$ for every $g \in R$. Furthermore, $\phi$ is a homeomorphism onto its image.
\end{lem}
\begin{proof}
Let $g \in R$. We have
\begin{align*}
\phi^{-1}(\widetilde{U_g}) &= \{P \in \Spec(R): \{R_f: f \notin P\} \cap U_g \neq \varnothing\}\\
&= \{P \in \Spec(R): \exists R_k \text{ with } R_k \ge R_g, k \notin P\}\\
&= \{P\in \Spec(R): \exists R_k \text{ with } k \notin P \text{ and } k \in rad(g)\}\\
&= \{P \in \Spec(R): g \notin P\}\\
&= D(g)
\end{align*}
and so $\phi$ is continuous.
A similar calculation yields $\phi(D(g))=\phi(\Spec(R)) \cap \widetilde{U_g}$ for every $g \in R$. Therefore, $\phi^{-1}: \phi(\Spec(R)) \rightarrow \Spec(R)$ is also continuous.
\end{proof}
Now recall that the structure sheaf on $\mathrm{Spec}(R)$ is the unique sheaf $\mathcal{O}_{\Spec(R)}$ with $\mathcal{O}_{\Spec(R)}(D(f))=loc(R,f)$ on the distinguished base and with the natural restriction maps. Therefore, for $f \in R$, $\mathcal{O}_{\Spec(R)}(D(f))$ and $\mathcal{O}_{\NCSpec(R)}(\widetilde{U_f})$ are defined by the same universal property, so there is a canonical isomorphism $$\phi^\#_{\widetilde{U_f}} : \mathcal{O}_{\NCSpec(R)}(\widetilde{U_f}) \rightarrow \mathcal{O}_{\Spec(R)}(D(f)).$$
Applying Lemma \ref{vakil2}, these maps glue together to give an isomorphism of sheaves
$$\phi^\#: \mathcal{O}_{\NCSpec(R)} \rightarrow \phi_* \mathcal{O}_{\Spec(R)}.$$
\begin{thrm}\label{commutative}
Let $R$ be a commutative ring. Then there is a morphism of ringed spaces
$$(\phi, \phi^\#): \Spec(R) \rightarrow \NCSpec(R)$$
 such that $\phi$ is a homeomorphism of $\Spec(R)$ with a dense subspace of $\NCSpec(R)\setminus\{\gamma\}$, where $\gamma$ is the generic point, and such that $\phi^\#$ is an isomorphism of sheaves on $\NCSpec(R)$.
\end{thrm}
\begin{proof}
We have shown everything except that the image of $\phi$ is dense in $\NCSpec(R) \setminus \{\gamma\}$. To show this, let $\widetilde{U_g}$ be a basic open subset of $\NCSpec(R)$ with $\widetilde{U_g}\neq\{\gamma\}$. Then $\widetilde{U_g}$ contains some point of $S(L(R))$ other than $\gamma$. Such a point corresponds to a proper closed subset $C$ of $L(R)$. Since $C \in \widetilde{U_g}$, $U_g \cap C \neq \varnothing$ and so $R_g \in C$. Therefore, $R_0 \neq R_g$ or else we would have $R_0 \in C$ and then $C=L(R)$. So $rad(g) \neq rad(0)$ and therefore there exists $P \in \Spec(R)$ with $g \notin P$. So $R_g \in \phi(P) \cap U_g$ and therefore $\phi(P)\cap U_g \neq \varnothing$ and so $\phi(P) \in \phi(\Spec(R))\cap \widetilde{U_g}$.
%
%
Thus $\phi(\Spec(R))$ meets every nonempty open subset of $\NCSpec(R) \setminus \{\gamma\}$ and so is dense in $\NCSpec(R) \setminus \{\gamma\}$.
\end{proof}
In Section \ref{aim}, it was claimed that $\Spec(R)$ embeds \emph{naturally} into $\NCSpec(R)$. By this is meant that the embeddings $(\phi,\phi^\#)$ of Theorem \ref{commutative} define a natural transformation
$$\Spec \rightarrow \NCSpec$$
of contravariant functors from the category of commutative rings to the category of ringed spaces. This in turn boils down to checking that if $R$ and $S$ are commutative rings and $\theta: R \rightarrow S$ is a ring homomorphism, then the map $\hat{\theta}$ of Theorem \ref{maintheorem} reduces to the usual map between prime spectra when restricted to the subspace $\Spec(S)\subset \NCSpec(S)$. This is easy to check.
\subsection{An exponential}
We wish to investigate the relationship between the spaces $\mathrm{Spec}(R)$ and $\NCSpec(R)$ in more detail, when $R$ is commutative. We can ask whether $\NCSpec(R)$ can be naturally viewed as a completion of $\mathrm{Spec}(R)$. Since completing something which is already complete should have no effect, a completion should be a functor $E: \mathbf{Top}\rightarrow \mathbf{Top}$ from the category of topological spaces to itself, such that $E^2=E$. It is natural to ask whether there is such a functor with the property $E(\mathrm{Spec}(R))=\NCSpec(R)$ for all commutative rings $R$. We will answer this question in the affirmative when $\mathbf{Top}$ is replaced by an appropriate category of based spaces, which we now describe.
\begin{defn}
If $X$ is a topological space, a base $\mathcal{B}$ of $X$ is called \emph{multiplicative} if $X \in \mathcal{B}$ and whenever $B_1, B_2 \in \mathcal{B}$, then $B_1 \cap B_2 \in \mathcal{B}$. We let $\mathfrak{T}$ denote the category whose objects are pairs $(X,\mathcal{B})$ where $X$ ia a $T_0$ space and $\mathcal{B}$ is a multiplicative base of $X$, and whose morphisms $(X,\mathcal{B})\rightarrow (Y,\mathcal{C})$ are functions $f: X \rightarrow Y$ such that $f^{-1}(C)\in\mathcal{B}$ for all $C \in \mathcal{C}$.
\end{defn}
Note that the category of $T_0$ spaces is a full subcategory of $\mathfrak{T}$ via $X \mapsto (X,\mathcal{T})$ where $X$ is a $T_0$ space with topology $\mathcal{T}$. Also, the category of sets is a full subcategory of $\mathfrak{T}$ via $X \mapsto (X,\mathcal{P}(X))$ where $\mathcal{P}(X)$ denotes the set of all subsets of $X$.

Now we define a functor $E: \mathfrak{T}\rightarrow\mathfrak{T}$ as follows. If $(X,\mathcal{B})$ is an object of $\mathfrak{T}$, and $B \in \mathcal{B}$, we set $\widetilde{B}_0=\{A \in \mathcal{P}(X): A \subset B\}$. Then $\widetilde{\mathcal{B}}=\{\widetilde{B}_0: B \in \mathcal{B}\}$ is a multiplicative base for a topology on $\mathcal{P}(X)$. We set $E_0(X, \mathcal{B})=(\mathcal{P}(X), \widetilde{\mathcal{B}})$. The space $\mathcal{P}(X)$ may not be $T_0$, so we define an equivalence relation on it by $A \sim A'$ if and only if $A$ and $A'$ belong to the same open subsets.
\begin{defn}
The \emph{exponential} $E(X,\mathcal{B})$ is the space $\mathcal{P}(X)/\sim$ equipped with the quotient topology, together with the multiplicative base $\widetilde{\mathcal{B}}=\{\widetilde{{B}}: B \in \mathcal{B}\}$ where $\widetilde{B}=\{[A] \in \mathcal{P}(X)/\sim \text{ such that } A \in \widetilde{B}_0\}$.
\end{defn}
It is easy to see that $E: \mathfrak{T} \rightarrow \mathfrak{T}$ is a functor and that if $(X,\mathcal{B})\in\mathfrak{T}$ then there is a morphism $\phi_{(X,\mathcal{B})}: (X,\mathcal{B}) \rightarrow E(X,\mathcal{B})$ defined by $\phi_{(X,\mathcal{B})}(x)=[\{x\}]$, the equivalence class of the singleton $\{x\}$. For any $(X,\mathcal{B})$, this map $\phi_{(X,\mathcal{B})}$ is an embedding whose image is dense in $E(X,\mathcal{B})\setminus\{[\varnothing]\}$, and the $\phi_{(X,\mathcal{B})}$ are the components of a natural transformation $\phi: id \rightarrow E$ of functors $\mathfrak{T} \rightarrow \mathfrak{T}$.

The following proposition explains the reason for defining $E$.
\begin{prop}\label{ncspec=expspec}
Let $R$ be a commutative ring. If $(X,\mathcal{B})=(\mathrm{Spec}(R),\{D(f): f \in R\})$ then $E(X,\mathcal{B})$ is naturally isomorphic to $(\NCSpec(R), \{\widetilde{U_f}: f \in R\})$.
\end{prop}
\begin{proof}
We have shown in Proposition \ref{unionofprimes} that the points of $\NCSpec(R)$ correspond to unions of prime ideals of $R$, and that $\widetilde{U_f}$ may be identified with $\{\bigcup_\lambda P_\lambda: f \notin \bigcup_\lambda P_\lambda\}$. We define a map $\gamma: E(X,\mathcal{B}) \rightarrow (\NCSpec(R), \{\widetilde{U_f}: f \in R\})$ by $\gamma : [A] \mapsto \bigcup_{P \in A} P$. The map $\gamma$ is well-defined because if $[A]=[A']$ then for all $f \in R$, if $f \notin \bigcup_{P \in A}P$ then $P \in D(f)$ for all $P \in A$, so $A \subset D(f)$ and therefore $A' \subset D(f)$ which implies $f \notin \bigcup_{P\in A'}P$. Therefore, $\bigcup_{P \in A'}P \subset\bigcup_{P\in A}P$, and the reverse inclusion is proved in the same way. Also, the map $\gamma$ is clearly surjective, and a similar argument to the above shows that $\gamma$ is injective. Also, for all $f \in R$, $\gamma^{-1}(\widetilde{U_f})=\widetilde{D(f)}$. Therefore, $\gamma$ is an isomorphism in $\mathfrak{T}$. It is straightforward to show that the ismorphisms $\gamma$ are natural with respect to $R$.
\end{proof}
Note that in the case $(X,\mathcal{B})=(\mathrm{Spec}(R), \{D(f)\})$, the map $\phi_{(X,\mathcal{B})}$ coincides with the map $\phi$ defined in Proposition \ref{phicts}.

We now wish to describe the functor $E$ via a universal property. Recall the notion of \emph{join-semlilattice} from Definition \ref{jsem}.
\begin{defn}
A join-semilattice $(X,\le)$ is called \emph{complete} if every subset $A$ of $X$ has a least upper bound, denoted $\mathrm{sup}(A)$. A morphism of join-semilattices is a (necessarily order-preserving) function $f: (X,\le_X)\rightarrow (Y,\le_Y)$ such that $f(\mathrm{sup}(A))=\mathrm{sup}(f(A))$ for all subsets $A$ of $X$.
\end{defn}
\begin{defn}
A $\mathfrak{T}$--complete join-semilattice is a triple $(X,\mathcal{B},\le)$ such that
\begin{itemize}
\item $(X,\mathcal{B})$ is an object of $\mathfrak{T}$.
\item $\le$ is a partial order on $X$ such that $(X,\le)$ is a complete join-semilattice.
\item The following property holds:

For all $A \subset X$ and all $B \in \mathcal{B}$, $A \subset B$ if and only if $\mathrm{sup}(A) \in B$.
\end{itemize}
\end{defn}
Observe that if $(X,\mathcal{B}) \in \mathfrak{T}$ then $E(X,\mathcal{B})$ is a complete join-semilattice under the order $$[A]\le [A'] \iff (\forall B \in \mathcal{B}, A \subset B \implies A' \subset B).$$ The reader can check that this order is well-defined and that the join of $\{[A_\lambda]: \lambda \in \Lambda\}$ is $\left[\bigcup_\lambda A_\lambda\right]$.
\begin{lem}\label{tcjsfull}
If $(X,\mathcal{B}, \le_X)$ and $(Y, \mathcal{C}, \le_Y)$ are $\mathfrak{T}$--complete join-semilattices and $f: (X,\mathcal{B}) \rightarrow (Y, \mathcal{C})$ is a morhpism in $\mathfrak{T}$, then $f$ is a morphism of complete join-semilattices.
\end{lem}
\begin{proof}
Let $A \subset X$. We must show that $\mathrm{sup}(f(A))=f(\mathrm{sup}(A))$. Since $Y$ is $T_0$, it suffices to show that for all $C \in \mathcal{C}$, $\mathrm{sup}(f(A)) \in C$ if and only if $f(\mathrm{sup}(A))\in C$. Let $C \in \mathcal{C}$. Then $f(\mathrm{sup}(A))\in C$ if and only if $\mathrm{sup}(A) \in f^{-1}(C)$ if and only if $A \subset f^{-1}(C)$ if and only if $f(A) \subset C$ if and only if $\mathrm{sup}(f(A)) \in C$, as required.
\end{proof}
\begin{prop}\label{univprop}
The map $\phi_{(X,\mathcal{B})} : (X,\mathcal{B}) \rightarrow E(X,\mathcal{B})$ is the universal map from $(X,\mathcal{B})$ to a $\mathfrak{T}$--complete join-semilattice, in the following sense:

$\phi_{(X,\mathcal{B})} : (X,\mathcal{B}) \rightarrow E(X,\mathcal{B})$ is a map from from $(X,\mathcal{B})$ to a $\mathfrak{T}$--complete join-semilattice, and if $\theta: (X,\mathcal{B}) \rightarrow (Y,\mathcal{C},\le)$ is any map from $(X,\mathcal{B})$ to a $\mathfrak{T}$--complete join-semilattice, then there exists a unique map of $\mathfrak{T}$--complete join-semilattices $\widehat{\theta} : E(X,\mathcal{B}) \rightarrow (Y,\mathcal{C},\le)$ such that the following diagram commutes.
$$
\xymatrix{
(X,\mathcal{B}) \ar[rr]^\theta \ar[d]_{\phi_{(X,\mathcal{B})}} &&(Y,\mathcal{C}, \le)\\
E(X,\mathcal{B}) \ar[rru]_{\widehat{\theta}}
}
$$
\end{prop}
\begin{proof}
The map $\widehat{\theta}$ is defined by $\widehat{\theta}([A])= \mathrm{sup}(\theta(A))$. It is routine to check that $\widehat{\theta}$ is the unique morphism in $\mathfrak{T}$ which makes the diagram commute. Lemma \ref{tcjsfull} then implies that $\widehat{\theta}$ is a morphism of $\mathfrak{T}$--complete join-semilattices.
\end{proof}
The universal property of Proposition \ref{univprop} makes it easy to prove that $E^2=E$.
\begin{prop}
$E^2=E$ as functors $\mathfrak{T}\rightarrow \mathfrak{T}$.
\end{prop}
\begin{proof}
Immediate from Proposition \ref{univprop} combined with Lemma \ref{tcjsfull}.
\end{proof}
We end this section with the following remark. We can define a category $\mathfrak{T}$--{\bf RingedSp} whose objects are objects of $\mathfrak{T}$ equipped with a sheaf of rings, together with the obvious morphisms. We can define a functor $\NCSpec_B: \mathbf{Rings}^{op}\rightarrow \mathfrak{T}-\mathbf{RingedSp}$ by $\NCSpec_B(R)=(\NCSpec(R), \{\widetilde{U_E} : E \in L_0(R)\})$. If $\mathbf{CommRings}$ denotes the category of commutative rings, then we can define $\mathrm{Spec}_B : \mathbf{CommRings}^{op} \rightarrow \mathfrak{T}-\mathbf{RingedSp}$ by $\mathrm{Spec}_B(R)=(\mathrm{Spec}(R), \{D(f): f \in R\})$. Finally, we can define $E: \mathfrak{T}-\mathbf{RingedSp} \rightarrow \mathfrak{T}-\mathbf{RingedSp}$ by extending the definition of $E$ to ringed spaces. These functors fit together in the following way.
\begin{prop}\label{ersatz}
The following diagram of categories and functors commutes.
$$
\xymatrix{
\mathbf{Rings}^{op} \ar[rr]^{\NCSpec_B} && \mathfrak{T}-\mathbf{RingedSp}\\
&&\\
\mathbf{CommRings}^{op} \ar[uu] \ar[rr]^{\Spec_B} && \mathfrak{T}-\mathbf{RingedSp} \ar[uu]_{E}
}
$$
\end{prop}
Proposition \ref{ersatz} is interesting because it is the closest we can get to the ideal situation of a functor $F: \mathbf{Rings}^{op} \rightarrow \mathbf{RingedSp}$ with the property that $F \cong \mathrm{Spec}$ when restricted to commutative rings.
\section{Examples}\label{examplessn}
In this section, we give some examples of $\NCSpec(R)$ for various rings $R$.
\begin{example}{\bf (The zero ring.)}\begin{rm}
Let $R=0$. According to our definitions, $L(R)$ consists of a single point, and therefore $\NCSpec(R)$ is the one-point space $\{\bullet\}$. This has a sheaf of rings whose sections over the two open sets $\varnothing$ and $\{\bullet\}$ are zero. By Theorem \ref{maintheorem}, for any $R$ there is a morphism $\NCSpec(0) \rightarrow \NCSpec(R)$. This morphism maps $\bullet$ to the generic point of $\NCSpec(R)$.
\end{rm}
\end{example}
\begin{rem}\begin{rm}
The reader may wonder why we do not simply delete the generic points. Indeed, if we defined $\NCSpec^0(R)$ to be the complement of the generic point in $\NCSpec(R)$, we would still obtain a faithful functor $\NCSpec^0: \Rings^{op}\rightarrow \mathbf{RingedSp}$. However, this definition would make it more awkward to write out some of the proofs of the main results of this paper, and would bring little benefit.
\end{rm}
\end{rem}
\begin{example}{\bf (Matrices.)}\label{matrix}\begin{rm}
Let $R=M_n(\mathbb{C})$, the ring of $n \times n$ matrices with complex entries. It is clear that if $E \subset R$ is a finite set of nonsingular matrices, then $loc(R,E)=R$. If $E \subset R$ is a finite set of matrices, at least one of which is singular, we claim that $loc(R,E)=0$. To see this, let $A \in E$ be a singular matrix and let $B$ be a matrix of rank $1$ with $\mathrm{image}(B) \subset \ker(A)$, and let $t \in loc(R,E)$ be the inverse of $\alpha_E(A)$. Then $ \alpha_E(B)=t\alpha_E(A)\alpha_E(B) = t \alpha_E(AB) =0$. But for each $ 1\le i \le n$, there is a nonsingular matrix $Z$ with $ZBZ^{-1}=E_{ii}$ where $E_{ii}$ is the matrix with zeroes everywhere except for a $1$ in the $(i,i)$--position. So $\alpha_E(E_{ii})=0$ for each $i$, and therefore $\alpha_E(1)=\alpha_E(\sum_i E_{ii})=0$. So $loc(R,E)=0$ as claimed.

Thus, the space $L(R)$ has only two elements, and so $S(L(R))$ has two points; the generic point $\gamma$ and one other point $p$ such that $\{p\}$ is closed. The nonempty open sets are $\{\gamma,p\}$ and $\{\gamma\}$, and the structure sheaf $\mathcal{O}$ has $\mathcal{O}(\{\gamma,p\})=R$, $\mathcal{O}(\{\gamma\})=0$ with the zero map as the restriction map.
\end{rm}
\end{example}
\begin{example}{\bf (Semisimple algebras.)}\label{ssa}\begin{rm}
Now let $A=M_{n_1}(\mathbb{C}) \times M_{n_2}(\mathbb{C})\times \cdots \times M_{n_k}(\mathbb{C})$ be a semisimple $\mathbb{C}$--algebra. For $1 \le i \le k$, write $e_i =(0,0, \ldots, 1, \ldots, 0)$, so that $\{e_i\}$ is a set of pairwise orthogonal central idempotents. For $I \subset \{1,2, \ldots, k\}$, write $A_I= A(\sum_{i \in I} e_i)$. Then $\sum_{i \in I} e_i$ is the identity element of $A_I$, and one can check that $A_I = loc(A, \sum_{i \in I} e_i)$ with the localisation map $A \rightarrow A_I$ being given by $a \mapsto a(\sum_{i \in I} e_i)$. We claim that every $loc(R,E)$ is canonically isomorphic to one of the $A_I$. To show this, we use the fact that for any ring $R$, if $x,y \in R$ and $xy=yx$ then $xy$ is a unit if and only if both $x$ and $y$ are units. If $E$ is a finite subset of $R$, suppose $\sum a_i e_i \in E$. Then $\alpha_E(\sum a_i e_i)$ is a unit in $loc(R,E)$. If $\det(a_i)=0$ then there exists $b_i \neq 0$ with $a_i b_i=0$ and so by the same argument as was used in Example \ref{matrix}, we obtain $\alpha_E(e_i)=0$. Thus, $\alpha_E(\sum a_ie_i)=\alpha_E(\sum_{\det(a_i)\neq 0}a_ie_i)$. Now using the fact that $(\sum_{\det(a_i)\neq 0} a_i^{-1} e_i) (\sum_{\det(a_i)\neq 0} a_i e_i)=(\sum_{\det(a_i)\neq 0} a_i e_i)(\sum_{\det(a_i)\neq 0} a_i^{-1} e_i)=\sum_{\det(a_i) \neq 0} e_i$, we obtain that $\alpha_E(\sum a_i e_i)$ is a unit if and only if $\alpha_E(\sum_{\det(a_i) \neq 0} e_i)$ is a unit. Finally, we obtain that $\alpha_E(e)$ is a unit for all $e \in E$ if and only if $\alpha_E(\sum_{i \in Z} e_i)$ is a unit, where
$$Z= \bigcap_{e \in E} \{ j: e=\sum a_j e_j \text{ with } \det(a_j) \neq 0\}.$$
Thus, $loc(R,E)=A_Z$. So the localization semilattice is isomorphic to the lattice of finite subsets of $\{1,2,\ldots, k\}$. From this it follows that the points of $\NCSpec(R)$ correspond to finite subsets of $\{1,2, \ldots, k\}$, and the basic open sets are of the form $U_I=\{J : J \subset I\}$, with $\mathcal{O}(U_I)=A_I$. If $U_{I_1}\subset U_{I_2}$ then $I_1 \subset I_2$ and the restriction map $A_{I_2} \rightarrow A_{I_1}$ is multiplication by $\sum_{i \in I_1}e_i$. It follows that for a general open $U$, $\mathcal{O}(U)=A_{\bigcup U}$.
\end{rm}
\end{example}
\begin{example}{\bf (Polynomials in one variable.)}\label{poly}\begin{rm}
Finally, let us consider $R=\mathbb{C}[x]$ (or, more generally, any PID).
Recall from Proposition \ref{unionofprimes} and Proposition \ref{ncspec=expspec} that points of $\NCSpec(R)$ correspond to unions $\bigcup_\lambda P_\lambda$ of prime ideals of $R$. Since $R$ is a PID and every nonzero ideal of $R$ is maximal, two such unions $\bigcup_\lambda P_\lambda$ and $\bigcup_\mu Q_\mu$ are equal if and only if
the nonzero ideals occuring in $\{P_\lambda\}$ are the same as those occuring in $\{Q_\mu\}$. The set of points of $\NCSpec(R)$ may therefore be identified with the zero ideal $\{\mathbf{0}\}$ together with the power set of the set of maximal ideals of $R$.

In the case $R=\mathbb{C}[x]$, 
the underlying space of $\NCSpec(R)$ is therefore $\mathcal{P}(\mathbb{C})\cup\{\mathbf{0}\}$ where $\mathcal{P}(\mathbb{C})$ denotes the set of all subsets of $\mathbb{C}$. If $f\neq 0$ then the basic open set $\widetilde{D(f)}$ is $\{A: A \subset \mathbb{C} \text{ and } f(a) \neq 0 \text{ for all } a \in A\}\cup\{\mathbf{0}\}$, while if $f=0$ then $\widetilde{D(f)}=\{\varnothing\}$. The generic point is $\varnothing \in \mathcal{P}(\mathbb{C})$.

A similar description also holds for any ring $R$ satisfying the property that if a prime ideal $P$ is contained in a union $\bigcup_\lambda P_\lambda$ of prime ideals, then $P \subset P_\lambda$ for some $\lambda$. Such rings were studied by Reis-Viswanathan \cite{Reis} and Smith \cite{WWSmith}.
\end{rm}
\end{example}
\section{Quasicoherent sheaves and glueing}\label{sheavesandglueing}
In algebraic geometry, one reason for defining affine schemes is so that larger spaces, such as projective space, can be built up by glueing together $\Spec(R)$ for various commutative rings $R$. It is natural to ask whether the spaces $\NCSpec(R)$ can also be glued together, and how much of algebraic geometry can be developed in this setting.
\subsection{Glueing}
Recall the notion of glueing ringed spaces. This can be found, for example in
\cite[Exercise 2.12]{Hartshorne} (it is stated there for schemes, but the same definitions can be made for ringed spaces in general).
\begin{prop}\label{glueingringedspaces}
Let $\{(X_\alpha, \mathcal{O}_{X_\alpha})\}_{\alpha \in A}$ be a collection of ringed spaces. Suppose for each $\alpha, \beta \in A$, we are given an open set $U_{\alpha \beta} \subset X_\alpha$, with $U_{\alpha \alpha}=X_\alpha$ for all $\alpha$, and suppose there are isomorphisms of ringed spaces $\varphi_{\alpha \beta} : U_{\alpha \beta} \rightarrow U_{\beta \alpha}$ such that $\varphi_{\alpha \alpha}$ is the identity for all $\alpha$, and $\varphi_{\beta \alpha}\varphi_{\alpha \beta}$ is the identity for all $\alpha$ and $\beta$.

Suppose further that $\varphi_{\alpha \beta}^\gamma:=\varphi_{\alpha \beta}|_{U_{\alpha \beta} \cap U_{\alpha \gamma} }$ maps $U_{\alpha \beta} \cap U_{\alpha \gamma}$ isomorphically to $U_{\beta \alpha} \cap U_{\beta \gamma}$ for all $\alpha, \beta, \gamma$ in $A$, and these isomorphisms satisfy
$$\varphi_{\gamma \beta}^{\alpha} \varphi_{\alpha\gamma}^\beta = \varphi_{\alpha \beta}^\gamma$$
for all $\alpha, \beta,\gamma \in A$.

Then there exists a ringed space $X$ defined as the quotient space of $\bigsqcup_\alpha X_\alpha$ by the equivalence relation generated by $x \sim \varphi_{\alpha \beta}(x)$ for all $x \in U_{\alpha\beta}$ and all $\alpha,\beta \in A$, together with the natural sheaf of rings. Also, $X$ has an open cover $X = \bigcup X_\alpha$ by the ringed spaces $X_\alpha$, in the sense that for each $\alpha$ there exists a morphism $\psi_\alpha : X_\alpha\rightarrow X$ which is an isomorphism of $X_\alpha$ with an open subspace of $X$, and $\psi_\alpha$ restricts to an isomorphism of $U_{\alpha\beta}$ with $\psi_\alpha(X_\alpha)\cap \psi_\beta(X_\beta)$ for all $\alpha$ and $\beta$, and $\psi_\alpha = \psi_\beta \varphi_{\alpha\beta}$ on $U_{\alpha \beta}$.
\end{prop}
In commutative algebra, the simplest case of Proposition \ref{glueingringedspaces} is when we have two commutative rings $R$ and $S$, and $loc(R,f)\cong loc(S,g)$ for some $f \in R$ and some $g \in S$. Then $\Spec(R)$ and $\Spec(S)$ may be glued along the open subset $D(f) \cong D(g)$. This construction cannot be generalized directly to noncommutative rings, because if $R$ is a noncommutative ring then it may not be true that the basic open subset $\widetilde{U_E}$ of $\NCSpec(R)$ is isomorphic as a ringed space to $\NCSpec(loc(R,E))$. However, we can identify a common situation in which it is.
\begin{defn}
Let $R$ be a ring. A subset $S$ of $R$ is \emph{multiplicatively closed} if $1 \in S$ and for all $s, s' \in R$, if $s, s' \in S$ then $ss' \in S$.
\end{defn}
\begin{defn}\cite[2.1.6]{McRob1987}\label{Oresetdef}
Let $R$ be a ring. A multiplicatively closed subset $S$ of $R$ is called a \emph{right Ore set} if whenever $r \in R$ and $s \in S$, there exist $r' \in R$ and $s' \in S$ with $rs'=sr'$.
\end{defn}
There is an analogous notion of left Ore set. The point of Definition \ref{Oresetdef} is that a localization of $R$ at an Ore set $S$ will have a nice form, because every fraction of the form $s^{-1}r$ can be rewritten as $r'(s')^{-1}$ with the inverse of an element of $S$ on the right.

We need one more definiton.
\begin{defn}
If $R$ is a ring and $E \subset R$, write $\langle E\rangle$ for the smallest multiplicatively closed subset of $R$ containing $E$. Thus, $\langle E\rangle$ consists of $1_R$ together with all the elements of $R$ formed by multiplying elements of $E$.
\end{defn}
Note that if $R$ is a ring and $E \subset R$ then $loc(R,E)=loc(R,\langle E\rangle)$. \begin{lem}\label{Oreset}
Let $R$ be a ring and let $E$ be a finite subset of $R$ such that $\langle E \rangle$ is a left or right Ore set. Let $U_E=\{R_F : R_F \ge R_E\}$ be the basic open subset of $L(R)$ corresponding to $E$. Let $\widetilde{U_E}$ be the corresponding basic open subset of $\NCSpec(R)$. Then there is an isomorphism of ringed spaces
$$\left(\widetilde{U_E}, \mathcal{O}_{\NCSpec(R)}|_{\widetilde{U_E}}\right) \cong \NCSpec(loc(R,E)).$$
\end{lem}
\begin{proof}
Write $R \stackrel{\alpha_E}{\longrightarrow} loc(R,E)$ for the localization of $R$ at $E$.
We wish to define an isomorphism of ringed spaces
$$\widetilde{U_E} \rightarrow \NCSpec(loc(R,E)).$$
First, we define a function $\phi:U_E \rightarrow L(loc(R,E))$. If $R_F \in U_E$, we set $\phi(R_F) = loc(R,E)_{\alpha_E(F)}$. By Proposition \ref{theta}, this is a well-defined order-preserving map. By Proposition \ref{thetaa}, the following diagram commutes.
$$
\xymatrix{
R \ar[rr]^{\alpha_E} \ar[d]_{\alpha_F} && loc(R,E)\ar[d]^{\alpha_{\alpha_E(F)}}\\
loc(R,F) \ar[rr]^{(\alpha_E)_F} && loc(loc(R,E), \alpha_E(F))
}
$$
To proceed further, we make the observation that if $R_F \in U_E$ then the map $(\alpha_E)_F$ in the above diagram is an isomorphism. To show this, observe that if $R_F \in U_E$ then $R_F \ge R_E$ so $F \succeq E$. Consider the following diagram.
$$\xymatrix{
R \ar[rr]^{\alpha_E} \ar[drr]^{\alpha_F} && loc(R,E) \ar[d]^{p_{FE}}\ar[rr]^>>>>>>{\alpha_{\alpha_E(F)}}&& loc(loc(R,E),\alpha_E(F))\\
&& loc(R,F) &&
}
$$
If $f \in F$ then $p_{FE}\alpha_E(f)=\alpha_F(f)$ which is a unit in $loc(R,F)$, so there is a unique $\delta : loc(loc(R,E), \alpha_E(F)) \rightarrow loc(R,F)$ such that $\delta \alpha_{\alpha_E(F)} =p_{FE}$. Then $\delta (\alpha_E)_F \alpha_F =\delta \alpha_{\alpha_E(F)} \alpha_E = p_{FE}\alpha_E =\alpha_F$. By Proposition \ref{epi}, we obtain $\delta (\alpha_E)_F = id$. On the other hand, $(\alpha_E)_F \delta \alpha_{\alpha_E(F)} \alpha_E = (\alpha_E)_F p_{FE}\alpha_E= (\alpha_E)_F\alpha_F =\alpha_{\alpha_E(F)}\alpha_E$. Since $\alpha_E$ and $\alpha_{\alpha_E(F)}$ are epic in the category of rings, we obtain $(\alpha_E)_F \delta=id$. So $(\alpha_E)_F$ is invertible as claimed.

Now, if $R_F, R_{F'} \in U_E$ with $\phi(R_F) \le \phi(R_{F'})$ then $\alpha_E(F) \preceq \alpha_E(F')$. So for all $f \in F$, $\alpha_{\alpha_E(F')}(\alpha_E(f))$ is a unit in $loc(loc(R,E) ,\alpha_E(F'))$. But $\alpha_{\alpha_E(F')}\alpha_E(f)=(\alpha_E)_{F'}\alpha_{F'}(f)$ so $\alpha_{F'}(f)$ is a unit because $(\alpha_E)_{F'}$ is an isomorphism. Therefore, $R_F\le R_{F'}$.

We have shown that $\phi(R_F) \le \phi(R_{F'})$ implies $R_F\le R_{F'}$. This implies that $\phi$ is one-to-one and that $\phi^{-1}$ is continuous where defined (since it is order-preserving).

Now we show that $\phi$ is surjective by using the assumption that $\langle  E \rangle$ is an Ore set. Suppose $\langle E \rangle$ is a right Ore set (the proof for a left Ore set is analogous). By \cite[Lemma 2.1.8]{McRob1987}, we may write every element of $loc(R,E)$ in the form $\alpha_E(g)(\alpha_E(e))^{-1}$ for some $g \in R$ and $e \in \langle E \rangle$. If $F=\{f_1, \ldots, f_n\}$ is an arbitrary finite subset of $loc(R,E)$, we write $f_i = \alpha_E(g_i)(\alpha_E(e_i))^{-1}$, $1 \le i \le n$. Let $G=\{g_1, \ldots, g_n\} \subset R$. Then $loc(R,E)_F=loc(R,E)_{\alpha_E(G)}=loc(R,E)_{\alpha_E(E \cup G)}$. So $R_{E \cup G} \in U_E$ and $\phi(R_{E \cup G})=loc(R,E)_F$. Therefore, $\phi$ is surjective. Since we have already shown that $\phi$ and $\phi^{-1}$ are continuous, $\phi$ is a homeomorphism.

The homeomorphism $\phi: U_E \rightarrow L(loc(R,E))$ induces a homeomorphism
$$\tilde{\phi} : S(U_E) \rightarrow S(L(loc(R,E))).$$
To obtain the desired homeomorphism $\widetilde{U_E} \rightarrow S(L(loc(R,E)))$, we need to check that $S(U_E)$ (the soberification of the subspace $U_E$ of $L(R)$) is homeomorphic to $\widetilde{U_E}$, a subspace of $S(L(R))$. To see this, we may define a map $S(U_E) \rightarrow S(L(R))$ by $C \mapsto \overline{C}=\{R_F \in L(R): R_F \le R_G \text{ for some } R_G \in C\}$. This map is continuous and its image is $\widetilde{U_E}$. We may define a continuous inverse via $Z \mapsto Z \cap U_E$ for $Z \in \widetilde{U_E} \subset S(L(R))$. Thus, $S(U_E)$ is homeomorphic to $\widetilde{U_E}$ and so we have a homeomorphism
$$\varphi: \widetilde{U_E} \rightarrow S(L(loc(R,E))).$$
We must now define a map of sheaves $\mathcal{O}_{S(L(loc(R,E)))}\rightarrow \varphi_*\mathcal{O}_{\widetilde{U_E}}$. We have already seen that any basic open subset of $L(loc(R,E))$ is of the form $U_{\alpha_E(F)}$ for some $F \in L_0(R)$ with $R_F \ge R_E$. So any basic open subset of $S(L(loc(R,E)))$ is of the form $\widetilde{U_{\alpha_E(F)}}$ for some finite $F \subset R$ with $R_F \ge R_E$. Using the definition of $\varphi$, we calculate
\begin{align*}
\varphi^{-1}(\widetilde{U_{\alpha_E(F)}})&= \{C \in \widetilde{U_E}: C\cap U_E \in \phi^{-1}(\widetilde{U_{\alpha_E(F)}})\}\\
&= \{C \in \widetilde{U_E} : \phi(C \cap U_E) \cap U_{\alpha_E(F)} \neq \varnothing\}\\
&= \{C \in \widetilde{U_E} : \exists R_A \in C\cap U_E \text{ with } loc(R,E)_{\alpha_E(A)} \ge loc(R,E)_{\alpha_E(F)}\}\\
&= \{C \in \widetilde{U_E} :\exists R_A \in C\cap U_E \text{ with } \phi(R_A) \ge \phi(R_F)\}\\
&= \{C \in \widetilde{U_E} :\exists R_A \in C\cap U_E \text{ with } R_A \ge R_F\}\\
&= \{C \in \widetilde{U_E} : C \cap U_F \neq \varnothing\}\\
&= \widetilde{U_F}.
\end{align*}
So for each $F$, we need a ring isomorphism $loc(loc(R,E), \alpha_E(F)) \rightarrow loc(R,F)$. We use the map $(\alpha_E)_F^{-1}$ defined above. If $\widetilde{U_F} \subset \widetilde{U_{F'}} \subset \widetilde{U_E}$ then Proposition \ref{theta} implies that the following diagram commutes.
$$\xymatrix{
loc(R,F') \ar[d]_{p_{FF'}} && \ar[ll]_{(\alpha_E)_{F'}^{-1}} loc(loc(R,E),\alpha_E(F')) \ar[d]^{p_{\alpha_E(F)\alpha_E(F')}}\\
loc(R,F) && \ar[ll]_{(\alpha_E)_{F}^{-1}} loc(loc(R,E), \alpha_E(F))
}$$
We may therefore apply Lemma \ref{vakil2} to get the desired isomorphism of sheaves
$\varphi^\# : \mathcal{O}_{S(L(loc(R,E)))}\rightarrow \varphi_*\mathcal{O}_{\widetilde{U_E}}$. Then $(\varphi, \varphi^\#)$ is an isomorphism of ringed spaces, as required.
\end{proof}
Note that if $E$ is a finite subset of $R$ such that $\langle E \rangle$ is a left or right Ore set then the map $(\varphi, \varphi^\#)$ defined in Lemma \ref{Oreset} makes the following diagram of ringed spaces and morphisms commute
$$
\xymatrix{
&& \NCSpec(R) \\
\widetilde{U_E} \ar[urr]^i \ar[rr]_>>>>>>{(\varphi,\varphi^\#)} &&\NCSpec(loc(R,E)) \ar[u]_\alpha
}$$
where $i$ is the inclusion and $\alpha$ is the map of ringed spaces induced by $\alpha_E: R \rightarrow loc(R,E)$.
\begin{example}\label{Oreexample}
\begin{rm}
In view of Lemma \ref{Oreset}, we may ``glue rings along Ore localizations". To be precise, suppose $R_1, \ldots, R_n$ are rings and for each $i$ let $\{E_{ij} : 1 \le j \le n\}$ be a collection of finite subsets of $R_i$ such that each $\langle E_{ij} \rangle$ is a right Ore set and such that $E_{ii}=\{1_{R_i}\}$ for all $i$. Suppose for each $i,j$ there are isomorphisms
$$\psi_{ij}: loc(R_i, E_{ij}) \rightarrow loc(R_j, E_{ji}),$$
and for each $i,j,k$, there are isomorphisms
$$\psi_{ij}^k : loc(R_i, E_{ij}\cup E_{ik}) \rightarrow loc(R_j, E_{ji}\cup E_{jk})$$ such that the following diagram commutes for all $i,j,k$.
$$\xymatrix{
loc(R_i, E_{ij})\ar[d]_{p_{{E_{ij}\cup E_{ik}, E_{ij}}}} \ar[rr]^{\psi_{ij}} && loc(R_j, E_{ji})\ar[d]^{p_{E_{ji} \cup E_{jk}, E_{ji}}}\\
loc(R_i, E_{ij}\cup E_{ik}) \ar[rr]^{\psi_{ij}^k} && loc(R_j, E_{ji}\cup E_{jk})
}$$
(Note that Proposition \ref{epi} implies that the $\psi_{ij}^k$ are uniquely determined by the $\psi_{ij}$.)
Assume that $\psi_{ii}=\psi_{ij}\psi_{ji} = id$ for all $i,j$ and
$$\psi_{jk}^i \psi_{ij}^k = \psi_{ik}^j$$
for all $i,j,k$.

It is clear that $\langle E_{ij} \cup E_{ik} \rangle$ is a right Ore set if $\langle E_{ij}\rangle$ and $\langle E_{ik}\rangle$ are, and so in view of Lemma \ref{Oreset}, all the assumptions of Proposition \ref{glueingringedspaces} are satisfied, and so we may glue the spaces $\NCSpec(R_i)$ to obtain a ringed space $X$ which has an open cover $X=\bigcup_{i=1}^n X_i$ with $X_i \cong\NCSpec(R_i)$ for each $i$.

We shall return to this example later.
\end{rm}
\end{example}
\subsection{Quasicoherent sheaves}
Now we explain how an $R$--module gives rise to a sheaf of $\mathcal{O}$--modules on $\NCSpec(R)$. We could work with left $R$--modules, right $R$--modules or $R-R$--bimodules. We choose to work with left $R$--modules.
\begin{lem}
Let $R$ be a ring and let $M$ be a left $R$--module. Let $(X,\mathcal{O}_X)=\NCSpec(R)$. Then there is a sheaf $\widetilde{M}$ of left $\mathcal{O}_X$--modules on $X$ with
$$\widetilde{M}(\widetilde{U_E})=loc(R,E)\otimes_R M$$
and the restriction maps $p_{FE}\otimes 1: loc(R,E)\otimes_R M \rightarrow loc(R,F)\otimes_R M$ for each inclusion $\widetilde{U_F}\subset \widetilde{U_E}$ of basic open subsets of $X$.
\end{lem}
\begin{proof}
Just as in the construction of the structure sheaf $\mathcal{O}_X$ in Section \ref{Alexandrov}, we may define $\widetilde{M}$ as in the statement of the lemma, and then apply Lemma \ref{vakil1} (in its form for sheaves of $\mathcal{O}_X$--modules) to conclude that there exists a unique sheaf of left $\mathcal{O}_X$--modules which agrees with $loc(R,E) \otimes_R M$ on $\widetilde{U_E}$. The condition $(*)$ of Lemma \ref{vakil1} holds automatically because every basic open set $\widetilde{U_E}$ is completely $\cup$--irreducible.
\end{proof}
It is convenient now to define an analogue of the usual notion of quasicoherent sheaf.
\begin{defn}\label{qcohdef}
Let $X$ be a ringed space and let $\mathcal{M}$ be a sheaf of left $\mathcal{O}_X$--modules. We say $\mathcal{M}$ is \emph{quasicoherent} if there exists an open cover $X = \bigcup_{i \in I} U_i$ such that for each $i\in I$, $(U_i , \mathcal{O}_X|_{U_i}) \cong \NCSpec(A_i)$ for some ring $A_i$, and such that for each $i$ there is a left $A_i$--module $M_i$ with $\mathcal{M}|_{U_i} \cong \widetilde{M_i}$.
\end{defn}
If $X$ is a ringed space, we denote by $\mathcal{O}_X-\mathrm{Mod}$ the abelian category of all sheaves of left $\mathcal{O}_X$--modules (a.k.a. left $\mathcal{O}_X$--modules) and module homomorphisms, and by $\mathrm{Qcoh}(X)$ the full subcategory of quasicoherent left $\mathcal{O}_X$--modules. In general, $\mathrm{Qcoh}(X)$ might not be an abelian category.
\begin{lem}
Let $R$ be a ring. If $\mathcal{M}$ is a quasicoherent sheaf of left $\mathcal{O}_X$--modules on $X=\NCSpec(R)$ then there exists a left $R$--module $M$ such that $\mathcal{M} \cong \widetilde{M}$.
\end{lem}
\begin{proof}
If $\{U_i\}$ is any open cover of $X$, then we must have $U_i=X$ for some $i$, because $X$ is completely $\cup$--irreducible (indeed, $R_1 \in L(R)$ is the minimum element of $L(R)$, and so $L(R)=U_1$ is completely $\cup$--irreducible by Proposition \ref{basic}, which implies that $S(L(R))$ is completely $\cup$--irreducible as well). So one of the rings $A_i$ in Definition \ref{qcohdef} satisfies $\NCSpec(A_i) \cong \NCSpec(R)$ and therefore $R \cong A_i$ and we may take $M=M_i$.
\end{proof}
The following proposition is an analogue of \cite[Corollary 5.5]{Hartshorne}.
\begin{prop}\label{qcohonaffine}
Let $X=\NCSpec(R)$. Then $M \mapsto \widetilde{M}$ and $\mathcal{M} \mapsto \Gamma(X,\mathcal{M})$ are inverse equivalences of categories between $R-\mathrm{Mod}$ and $\mathrm{Qcoh}(X)$.
\end{prop}
\begin{proof}
The proposition is easy to check, because any morphism $\widetilde{M} \rightarrow \widetilde{M'}$ of sheaves of left $\mathcal{O}_X$--modules is determined by the corresponding map on global sections.
\end{proof}
\begin{rem}\begin{rm}
Proposition \ref{qcohonaffine} shows that $\mathrm{Qcoh}(\NCSpec(R))$ is an abelian category and that the functor $M \mapsto \widetilde{M}$ is exact. The reader may wonder how this is possible, since it is clear that if $M$ is a left $R$--module then the stalk of $\widetilde{M}$ over the point $\overline{\{R_E\}} \in \NCSpec(R)$ is $loc(R,E) \otimes_R M$, and yet if $0\rightarrow M'' \rightarrow M \rightarrow M' \rightarrow 0$ is a short exact sequence of left $R$--modules, then the sequence of stalks
$$0 \rightarrow loc(R,E) \otimes_R M'' \rightarrow loc(R,E)\otimes_R M \rightarrow loc(R,E) \otimes_R M' \rightarrow 0$$
need not be exact, because $loc(R,E)$ need not be a flat $R$--module. However, the point is that $M \mapsto \widetilde{M}$ is \emph{not} exact when viewed as a functor from $R-\mathrm{Mod}$ to $\mathcal{O}_X-\mathrm{Mod}$. It is only exact when viewed as a functor $R-\mathrm{Mod} \rightarrow \mathrm{Qcoh}(X)$. Kernels and cokernels in $\mathrm{Qcoh}(X)$ may not agree with kernels and cokernels in $\mathcal{O}_X-\mathrm{Mod}$. For example, if $f : \widetilde{M} \rightarrow \widetilde{N}$ is a morphism, then the kernel of $f$ in $\mathrm{Qcoh(X)}$ is $\widetilde{\ker(f_X)}$ where $f_X$ is the map induced by $f$ on global sections. But this is not in general the same as the sheaf kernel of $f$.
\end{rm}\end{rem}
\subsection{}
Now let $X$ be any ringed space which has a finite open cover $X=\bigcup_{i=1}^n X_i$ with $X_i \cong \NCSpec(R_i)$ for some rings $R_i$. Suppose this cover is irredundant; that is, no $X_i$ is contained in $\bigcup_{j \neq i}X_j$. Let $\mathcal{M}$ be a quasicoherent sheaf of left $\mathcal{O}_X$--modules. Then there is an open cover $\{U_j\}$ of $X$ with $U_j \cong \NCSpec(A_j)$ and $A_j$--modules $M_j$ with $\mathcal{M}|_{U_j} \cong \widetilde{M_j}$. Then since $X_i =\bigcup(X_i \cap U_j)$ for each $i$, but $X_i$ is completely $\cup$--irreducible, we have $X_i \subset U_j$ for some $j$. The same argument applied to $U_j$ shows that $U_j \subset X_k$ for some $k$. Therefore, each $U_j$ is equal to one of the $X_i$. Since the cover $\{X_i\}$ is irredundant, we must have $\{X_i\}_{i=1}^n \subset \{U_j\}$. Thus, for each $1 \le i \le n$, there is a left $R_i$--module $M_i$ with $\mathcal{M}|_{X_i} =\widetilde{M_i}$, and for each $i,j$, there is an isomorphism of sheaves on $X_i \cap X_j$
$$\varphi_{ij} : \widetilde{M_i}|_{X_i \cap X_j} \rightarrow \widetilde{M_j}|_{X_i \cap X_j}$$
satisfying
$$(**) \begin{cases}
\varphi_{ii}=id & \text{on } X_i\\
\varphi_{jk}\varphi_{ij}=\varphi_{ik} & \text{on } X_i \cap X_j \cap X_k\end{cases}$$
for all $i, j, k$.

Conversely, by \cite[Exercise 1.22]{Hartshorne}, the data $(\{M_i\}_{i=1}^n, \{\varphi_{ij}\})$, where $M_i$ is a left $R_i$--module for each $i$ and $\varphi_{ij} : \widetilde{M_i}|_{X_i \cap X_j} \rightarrow \widetilde{M_j}|_{X_i \cap X_j}$ are isomorphisms of sheaves satisfying $(**)$, uniquely determine a sheaf $\mathcal{M}$ of left $\mathcal{O}_X$--modules satisfying $\mathcal{M}|_{X_i} = \widetilde{M_i}$ for each $i$.

Thus, there is an equivalence of categories between the category of quasicoherent sheaves of left $\mathcal{O}_X$--modules and the category of data $(\{M_i\}, \{\varphi_{ij}\})$, together with the obvious notion of morphism $(\{M_i\}, \{\varphi_{ij}\})\rightarrow (\{M_i'\}, \{\varphi_{ij}'\})$.
\begin{example}\label{Oreexampleqcoh}
\begin{rm}
Consider quasicoherent sheaves on the space $X$ constructed in Example \ref{Oreexample}. We use the notation of that example. The space $X$ has an open cover by $X_i =\NCSpec(R_i)$ and $X_i \cap X_j \cong \NCSpec(loc(R_i, E_{ij})) \cong \NCSpec(loc(R_j, E_{ji}))$ for each $i, j$. By the above discussion, the category of quasicoherent sheaves on $X$ is equivalent to the category whose objects are data of the form
$$(\{M_i\}, \{\varphi_{ij}\})$$
where $M_i$ is a left $R_i$--module, $1 \le i \le n$, and the $\varphi_{ij}$ are coherence isomorphisms satisfying the following axioms.

For each $1 \le i,j \le n$,
$$\varphi_{ij}: loc(R_i, E_{ij}) \otimes_{R_i} M_i \rightarrow loc(R_j,E_{ji}) \otimes_{R_j} M_j$$
is an isomorphism of abelian groups such that $\varphi_{ij}(rm)=\psi_{ij}(r) \varphi_{ij}(m)$ for all $m \in loc(R_i, E_{ij}) \otimes_{R_i}M_i$ and all $r \in loc(R_i, E_{ij})$, and such that $\varphi_{ij}\varphi_{ji}=\varphi_{ii}=id$ for all $i$ and $j$.

Furthermore, for each $1 \le i,j,k \le n$, the unique morphism of abelian groups
$$\varphi_{ij}^k : loc(R_i, E_{ij} \cup E_{ik})\otimes_{R_i}M_i \rightarrow
loc(R_j, E_{ji} \cup E_{jk}) \otimes_{R_j} M_j$$
which satisfies  $\varphi_{ij}^k(rm)=\psi_{ij}^k(r) \varphi_{ij}^k(m)$ for all $r \in loc(R_i, E_{ij} \cup E_{ik})$ and all $m \in  loc(R_i, E_{ij} \cup E_{ik})\otimes_{R_i}M_i$, and which makes the following diagram commute,
$$
\xymatrix{
loc(R_i, E_{ij}) \otimes_{R_i} M_i \ar[d]_{p_{E_{ij} \cup E_{ik}, E_{ij}}\otimes 1} \ar[rr]^{\varphi_{ij}} && loc(R_j,E_{ji}) \otimes_{R_j} M_j\ar[d]^{p_{E_{ji} \cup E_{jk}, E_{ji}}\otimes 1}\\
loc(R_i, E_{ij} \cup E_{ik})\otimes_{R_i}M_i \ar[rr]^{\varphi_{ij}^k} &&
loc(R_j, E_{ji} \cup E_{jk}) \otimes_{R_j} M_j
}$$
is an isomorphism, and these isomorphisms satisfy $\varphi_{jk}^i \varphi_{ij}^k =\varphi_{ik}^j$ for all $i,j,k$.

Morphisms $(\{M_i\}, \{\varphi_{ij}\})\rightarrow (\{M_i'\}, \{\varphi_{ij}'\})$ are collections of maps $(f_i : M_i \rightarrow M_i')$ which commute with the $\varphi_{ij}$.
\end{rm}
\end{example}
\begin{example}{\bf (Quasicoherent sheaves on a noncommutative projective space.)}\begin{rm}
As an example of the notions discussed in Section \ref{sheavesandglueing}, we define an analogue of projective space over a ``polynomial ring" $R$ in skew-commuting variables, and we relate quasicoherent sheaves on this space to graded $R$--modules.

Let $n \in \mathbb{N}$ and for $1 \le i < j \le n$, let $\lambda(i,j)$ be a nonzero complex number. Let $R=\mathbb{C}_\lambda[x_1,x_2, \ldots, x_n]$ be the $\mathbb{C}$--algebra generated by $x_1, x_2, \ldots, x_n$ subject to the relations
$$x_i x_j = \lambda(i,j) x_j x_i$$
for $i <j$. Then $R$ is a graded ring with $x_1, x_2, \ldots, x_n$ in degree $1$, and for each $i$, the powers of $x_i$ form an Ore set in $R$. We denote $loc(R,x_i)$ by $R[x_i^{-1}]$. This is again a graded ring, and we write $R_i$ for the zeroth graded piece $R[x_i^{-1}]_0$. For each $i$ and $j$, the powers of $x_jx_i^{-1}$ form an Ore set in $R_i$, and there is a canonical isomorphism $$\psi_{ij}:R[x_i^{-1}]_0[(x_jx_i^{-1})^{-1}]\rightarrow R[x_j^{-1}]_0[(x_ix_j^{-1})^{-1}],$$ since both of these rings may be identified with the zeroth graded component of the localization of $R$ at $x_i x_j$. In what follows, we suppress the isomorphisms $\psi_{ij}$ and write $R_{ij} = R_{ji} = R[x_i^{-1}]_0[(x_jx_i^{-1})^{-1}]$. We also write $R_{ijk}$ for the zeroth graded component of the localization of $R$ at the set of powers of $x_ix_jx_k$.

We see that the set of rings $R_i$ together with the subsets $E_{ij}=\{x_jx_i^{-1}\}$ satisfy all the hypotheses of Example \ref{Oreexample} and we may therefore glue the $\NCSpec(R_i)$ together along $\NCSpec(R_{ij})$ to obtain a ringed space which we denote by $X$.

We wish to study quasicoherent sheaves on $X$. By Example \ref{Oreexampleqcoh}, a quasicoherent sheaf $\mathcal{M}$ on $X$ consists of the data
$$\mathcal{M} = (\{M_i\} , \{\varphi_{ij}\})$$
where $M_i$ is an $R_i$--module for each $i$, and
$$\varphi_{ij} : R_{ij} \otimes_{R_i} M_i \rightarrow R_{ij} \otimes_{R_j} M_j$$
is an isomorphism of $R_{ij}$--modules, with $\varphi_{ii}=\varphi_{ij}\varphi_{ji}=id$ for all $i,j$, and $\varphi_{jk}\varphi_{ij}=\varphi_{ik}$ as morphisms $R_{ijk} \otimes_{R_i} M_i \rightarrow R_{ijk}\otimes_{R_k} M_k$, for all $i,j,k$. A morphism $f: (\{M_i\}, \{\varphi_{ij}\}) \rightarrow (\{M_i'\},\{\varphi_{ij}'\})$ of quasicoherent sheaves consists of an $R_i$--module map $f_i : M_i \rightarrow M_i'$ for each $i$, with the property that the following diagram commutes for all $i$ and $j$, where the vertical maps are those induced by $f$.
$$\xymatrix{
R_{ij} \otimes_{R_i} M_i \ar[d]_{1 \otimes f_i} \ar[rr]^{\varphi_{ij}} && R_{ij} \otimes_{R_j} M_j \ar[d]^{1\otimes f_j} \\
R_{ij} \otimes_{R_i} M_i' \ar[rr]^{\varphi_{ij}'} &&R_{ij} \otimes_{R_j} M_j'
}$$
By \cite[2.1.16(ii)]{McRob1987}, $R_{ij}$ is a flat right and left $R_i$-- and $R_j$--module for all $i,j$. Therefore, the category $\mathrm{Qcoh}(X)$ is abelian. We wish to relate $\mathrm{Qcoh}(X)$ to graded left $R$--modules, by copying arguments due to Serre from algebraic geometry. We follow very closely the exposition in \cite[Class 30]{vakil}.

Let $R-\mathrm{GrMod}$ denote the category of $\mathbb{Z}$--graded left $R$--modules and graded maps. First, we need to construct a functor $R-\mathrm{GrMod}\rightarrow \mathrm{Qcoh}(X)$. Let $M=\bigoplus_{n \in \mathbb{Z}}M_n$ be a graded left $R$--module. We define a quasicoherent sheaf $\widetilde{M}=(\{\widetilde{M}_i\}, \{\varphi_{ij}\})$ on $X$ by
$$\widetilde{M}_i = (R[x_i^{-1}]\otimes_R M)_0,$$
the zeroth graded component of the graded left $R$--module $R[x_i^{-1}]\otimes_R M$. We define $\varphi_{ij} : R_{ij} \otimes_{R_i}\widetilde{M_i} \rightarrow R_{ij}\otimes_{R_j}\widetilde{M}_j$ by
$$\varphi_{ij}(a \otimes r_{-k} \otimes m_k)=ar_{-k}x_j^k\otimes x_j^{-k} \otimes m_k$$
where $m_k \in M_k$, $r_{-k} \in R[x_i^{-1}]_{-k}$ and $a \in R_{ij}$. Then $\varphi_{ij}$ is a well-defined $R_{ij}$--module isomorphism and the axioms for a quasicoherent sheaf are satisfied. Furthermore, $M \mapsto \widetilde{M}$ is an exact functor $R-\mathrm{GrMod}\rightarrow \mathrm{Qcoh}(X)$ (exactness follows from the flatness of $R[x_i^{-1}]$ over $R$).

Now suppose $\mathcal{M}=(\{M_i\},\{\varphi_{ij}\})$ is a quasicoherent sheaf on $X$. Let $R[x_i^{-1}]_n$ denote the $n^{th}$ graded piece of the graded $R$--module $R[x_i^{-1}]$. We define another quasicoherent sheaf $\mathcal{M}[n]=(\{\mathcal{M}[n]_i\}, \{\varphi_{ij}[n]\})$ by
$$\mathcal{M}[n]_i=R[x_i^{-1}]_n \otimes_{R_i} M_i$$
which is a left $R_i$--module, and we define $\varphi_{ij}[n]$ as follows. We may identify $R_{ij} \otimes_{R_i} R[x_i^{-1}]_n$ with $R[x_i^{-1},x_j^{-1}]_n$ via multiplication, and so we need to define an isomorphism
$$\varphi_{ij}[n]: R[x_i^{-1},x_j^{-1}]_n \otimes_{R_i} M_i \rightarrow R[x_i^{-1},x_j^{-1}]_n \otimes_{R_j} M_j$$
for each $i$ and $j$. For $r \in R[x_i^{-1},x_j^{-1}]_n$ and all $m \in M_i$, we set $\varphi_{ij}[n](r \otimes m) = x_i^n \varphi_{ij} (x_i^{-n} r \otimes m)$. To show that $\mathcal{M}[n]$ is a quasicoherent sheaf, we first check that $\varphi_{ij}[n]$ is a map of $R_{ij}$--modules. This is clear because $\varphi_{ij}$ is a map of $R_{ij}$--modules and therefore $x_i^n \varphi_{ij} (x_i^{-n} r \otimes m)= r \varphi_{ij}(1\otimes m)$.

Since $\varphi_{ij}$ is an $R_{ij}$--module map, the identity $\varphi_{ji}[n]\varphi_{ij}[n]=\varphi_{ii}[n]=id$ holds for all $i$ and $j$, and also $\varphi_{jk}[n]\varphi_{ij}[n]=\varphi_{ik}[n]$ as maps of $R_{ijk}$--modules, for all $i,j,k$. So $\mathcal{M}[n]$ is a quasicoherent sheaf.

We can make $\mathcal{M} \mapsto \mathcal{M}[n]$ into a functor by setting, for $f: \mathcal{M} \rightarrow \mathcal{N}$, $f[n]_i = x_i^n f x_i^{-n} : \mathcal{M}[n]_i \rightarrow \mathcal{N}[n]_i$.

For a quasicoherent sheaf $\mathcal{M}$, we define
$$\Gamma(\mathcal{M})= \bigoplus_{n \in \mathbb{Z}} \Gamma(X,\mathcal{M}[n])$$
which is a graded $R$--module.

Now let $\mathcal{M}=(\{M_i\}, \{\varphi_{ij}\})$ be a quasicoherent sheaf on $X$. We show that $\widetilde{\Gamma(\mathcal{M})} \cong \mathcal{M}$. It suffices to show this on each $X_i$ in the open cover of $X$, so it suffices to give compatible isomorphisms $f_i : \widetilde{\Gamma(\mathcal{M})}_i \rightarrow M_i$. By definition,
$$\widetilde{\Gamma(\mathcal{M})}_i=\left(R[x_i^{-1}]\otimes_R \bigoplus_{n \in \mathbb{Z}} \Gamma(X, \mathcal{M}[n]) \right)_0 =\sum_{j \in \mathbb{Z}} R[x_i^{-1}]_j \otimes_R \Gamma(X, \mathcal{M}[-j]).$$
For $a \otimes m \in R[x_i^{-1}]_j \otimes_R \Gamma(X, \mathcal{M}[-j])$, we define $f_i  (a \otimes m) = a\cdot m|_{X_i}$. This defines a map of $R_i$--modules $\widetilde{\Gamma(\mathcal{M})}_i \rightarrow M_i$. To show that this map is surjective, let $m \in M_i$. 
For $N \ge 0$, $\varphi_{ij}[N](x_i^N \otimes m)=x_i^N \varphi_{ij} (1 \otimes m) \in x_i^N R_{ij} \otimes_{R_j} M_j$.
It follows that for $N$ large enough, $x_i^Nm$ is the restriction to $X_i$ of a global section $s$ of $\mathcal{M}[N]$. Then $f(x_i^{-N}\otimes s)=m$ and so $f_i$ is surjective. To show that $f_i$ is injective, suppose $\sum_k a_k \otimes m_k \in \ker(f_i)$.
Then $\sum_k a_k\cdot  m_k |_{X_i}=0$. Choose $N$ large enough so that $x_i^{N}a_k \in R$ for all $k$. Let $u=\sum_k (x_i^Na_k) m_k$, a global section of $\mathcal{M}[N]$. Then for all $j$, $u|_{X_j} \in M_j$ maps to $0$ in $R_{ij}\otimes_{R_j} M_j$. By \cite[2.1.17]{McRob1987}, there exists $N_1 \in \mathbb{N}$ with $x_i^{N_1} (u|_{X_j})=0 \in M_j$ for all $j$. Therefore, $x_i^{N_1} u \in \Gamma(X,\mathcal{M}[N+N_1])$ is the zero section. Then $\sum_k a_k \otimes m_k=\sum_k x_i^{-N_1-N}x_i^{N_1+N}a_k \otimes m_k=x_i^{-N_1-N} \otimes x_i^{N_1+N} \sum_k a_k m_k=x_i^{-N_1-N} \otimes x_i^{N_1} u=0$.
Therefore, $f_i$ is the required isomorphism of $R_i$--modules $\widetilde{\Gamma(\mathcal{M})}_i \rightarrow M_i$. The $f_i$ agree on $X_i \cap X_j$ and so we obtain an isomorphism of sheaves $\widetilde{\Gamma(\mathcal{M})} \rightarrow \mathcal{M}$.

Next, we show that if $M \in R-\mathrm{GrMod}$ then there is a natural map $M \rightarrow \Gamma(\widetilde{M})$. We need to define a map $M_n \rightarrow \Gamma(X, \widetilde{M}[n])$ for each $n \in \mathbb{Z}$. On $X_i$, $\widetilde{M}[n]$ corresponds to the $R_i$--module $R[x_i^{-1}]_n \otimes_{R_i} (R[x_i^{-1}]\otimes_R M)_0$. For $m \in M_n$, we define $m_i: = x_i^n\otimes x_i^{-n} \otimes m \in R[x_i^{-1}]_n \otimes_{R_i} (R[x_i^{-1}]\otimes_R M)_0$. Then $\varphi_{ij}[n](m_i)=m_j$ for all $i$ and $j$, so the $m_i$ glue together to give a global section $\overline{m}$ of $\widetilde{M}[n]$. Then $m \mapsto \overline{m}$ defines a map
$\gamma_M : M \rightarrow \Gamma(\widetilde{M})$
of graded $R$--modules. These maps are natural, meaning that they define a natural transformation $id \rightarrow \Gamma \circ \widetilde{(\cdot)}$ of functors $R-\mathrm{GrMod} \rightarrow R-\mathrm{GrMod}$.

Now we define a subcategory $R-\mathrm{Tors}$ of $R-\mathrm{GrMod}$ to be the full subcategory of those objects $M$ such that for all $m \in M$ there exists $N\in\mathbb{N}$ such that $R_d m =0$ for all $d \ge N$. Note that if $M \in R-\mathrm{GrMod}$ then $\ker(\gamma_M)$ and $\mathrm{cok}(\gamma_M)$ belong to $R-\mathrm{Tors}$. Note also that if $M \in R-\mathrm{Tors}$ then $\widetilde{M}=0$.

Now recall (see for example \cite{popescu}) that the quotient category $\frac{R-\mathrm{GrMod}}{R-\mathrm{Tors}}$ is defined by the following universal property:

$\frac{R-\mathrm{GrMod}}{R-\mathrm{Tors}}$ is an additive category and there exists an additive functor $Q: R-\mathrm{GrMod} \rightarrow \frac{R-\mathrm{GrMod}}{R-\mathrm{Tors}}$ such that $Q(a)$ is an isomorphism for every arrow $a$ with $\ker(a), \mathrm{cok}(a) \in R-\mathrm{Tors}$, and furthermore if $\mathcal{C}$ is an additive category and $F: R-\mathrm{GrMod} \rightarrow\mathcal{C}$ is an additive functor such that $F(a)$ is an isomorphism for all $a$ with $\ker(a), \mathrm{cok}(a) \in R-\mathrm{Tors}$, then there exists a unique additive functor $G: \frac{R-\mathrm{GrMod}}{R-\mathrm{Tors}} \rightarrow \mathcal{C}$ such that the following diagram commutes.
$$\xymatrix{
R-\mathrm{GrMod} \ar[d]_{F} \ar[r]^<<<Q & \frac{R-\mathrm{GrMod}}{R-\mathrm{Tors}} \ar[ld]^G\\
\mathcal{C} &
}$$
We claim that the functor $Q: R-\mathrm{GrMod} \rightarrow \mathrm{Qcoh}(X)$ defined by $Q(M)=\widetilde{M}$ satisfies this universal property. Indeed, if $F: R-\mathrm{GrMod} \rightarrow \mathcal{C}$ is an additive functor such that $F(a)$ is an isomorphism for all $a$ with $\ker(a), \mathrm{cok}(a) \in R-\mathrm{Tors}$, we may define $G(\mathcal{M})=F(\Gamma(\mathcal{M}))$, and then for $M \in R-\mathrm{GrMod}$, $G(\widetilde{M})=F({\Gamma(\widetilde{M})})\cong F(M)$. After checking the appropriate naturality conditions, the following proposition can be obtained.
\begin{prop}\label{qcoh=grmod/tors}
The category $\mathrm{Qcoh}(X)$ is equivalent to $\frac{R-\mathrm{GrMod}}{R-\mathrm{Tors}}$. \end{prop}
Proposition \ref{qcoh=grmod/tors} is interesting because there is a philosophy due to Artin and Zhang \cite{ArtinZhang} of regarding the category $\frac{R-\mathrm{GrMod}}{R-\mathrm{Tors}}$ over a graded ring $R$ as the category of quasicoherent sheaves on some space $\mathrm{Proj}(R)$ associated to $R$, without necessarily assuming that such a space exists. 
Proposition \ref{qcoh=grmod/tors} shows that, in the given example, $\frac{R-\mathrm{GrMod}}{R-\mathrm{Tors}}$ may be identified with a category of sheaves on a ringed space which is an analogue of the usual
commutative $\mathrm{Proj}$.
\end{rm}
\end{example}
\section{Conclusion and questions}\label{conclusion}
In Section \ref{sheavesandglueing}, we have begun to attempt to generalize commutative algebraic geometry to the spaces $\NCSpec(R)$. This is an interesting project, and we hope to return to it in future work. However, it may be that the spaces $\NCSpec(R)$ are not suitable for doing geometry. In Example \ref{ssa} and Proposition \ref{ncspec=expspec}, we see that $\NCSpec(R)$ has more points than we probably really want. This is not surprising, because what we have really done is to construct a poset which we want to be the lattice of open subsets of a topological space (a \emph{locale}), and then we have added points in an abstract way in order to make it into a space. It may be that a more profitable direction of research would be to look for spaces for which Theorem \ref{maintheorem} holds, but which are smaller than $\NCSpec(R)$ and therefore more manageable. It may even be that there is a way to describe topologically a subspace of $\NCSpec(R)$ which corresponds to $\Spec(R)$ in the case when $R$ is commutative.

Another interesting question is to study whether the construction $\NCSpec(R)$ is related to the many other versions of noncommutative affine scheme which have been invented, some of which have had successful applications in representation theory. We discuss a few of these below, but there are many more which we do not mention. See for example \cite{Skoda} for a survey.

We mention in particular the \emph{schematic algebras} defined by Van Oystaeyen in \cite{VanObook}, and the noncommutative schemes defined by Rosenberg \cite{Rosenbergbook}, \cite{Rosenberg98}, and Kontsevich-Rosenberg \cite{KR}. We mention these notions in particular because \cite[Theorem 2.1.5]{VanObook} is a much more general version of Proposition \ref{qcoh=grmod/tors}, and the main theorem of \cite[Section 4.2]{Rosenberg98} resembles Theorem \ref{maintheorem2}. Our $\NCSpec(R)$ may also be related to a construction which appears in unpublished work of Schofield and which is mentioned in \cite[Chapter 7]{Schofieldbook}.
Other constructions which resemble ours to some extent include those of \cite{RosenbergLevitzki} and \cite{PMCohn}. Our construction differs from many of these in two ways. First, it is less algebraic and requires only the notion of abstract localization to construct, and second, we are able to prove Theorem \ref{maintheorem}. However, our $\NCSpec$ suffers from many of the same problems as earlier constructions. In particular, it seems to be very difficult to calculate explicit examples.

\bibliographystyle{alpha}

\begin{thebibliography}{Kap70}

\bibitem[AZ94]{ArtinZhang}
M.~Artin and J.~J. Zhang.
\newblock Noncommutative projective schemes.
\newblock {\em Adv. Math.}, 109(2):228--287, 1994.

\bibitem[Coh72]{PMCohn}
P.~M. Cohn.
\newblock Skew fields of fractions, and the prime spectrum of a general ring.
\newblock In {\em Lectures on rings and modules ({T}ulane {U}niv. {R}ing and
  {O}perator {T}heory {Y}ear, 1970-1971, {V}ol. {I})}, pages 1--71. Lecture
  Notes in Math., Vol. 246. Springer, Berlin, 1972.

\bibitem[Coh06]{Cohn}
P.~M. Cohn.
\newblock Localization in general rings, a historical survey.
\newblock In {\em Non-commutative localization in algebra and topology}, volume
  330 of {\em London Math. Soc. Lecture Note Ser.}, pages 5--23. Cambridge
  Univ. Press, Cambridge, 2006.

\bibitem[EH00]{EisenbudHarris}
D. Eisenbud and J. Harris.
\newblock {\em The geometry of schemes}, volume 197 of {\em Graduate Texts in
  Mathematics}.
\newblock Springer-Verlag, New York, 2000.

\bibitem[Gro60]{EGA1}
A.~Grothendieck.
\newblock \'{E}l\'ements de g\'eom\'etrie alg\'ebrique. {I}. {L}e langage des
  sch\'emas.
\newblock {\em Inst. Hautes \'Etudes Sci. Publ. Math.}, (4):228, 1960.

\bibitem[Har77]{Hartshorne}
R. Hartshorne.
\newblock {\em Algebraic geometry}.
\newblock Springer-Verlag, New York, 1977.
\newblock Graduate Texts in Mathematics, No. 52.

\bibitem[Hof79]{Hoffmann}
R. E. Hoffmann.
\newblock Sobrification of partially ordered sets
\newblock {\em Semigroup Forum}, 17(2):123-138, 1979.

\bibitem[Joh82]{Stone}
P.~T. Johnstone.
\newblock {\em Stone spaces}, volume~3 of {\em Cambridge Studies in Advanced
  Mathematics}.
\newblock Cambridge University Press, Cambridge, 1982.

\bibitem[Kap70]{Kaplansky}
I. Kaplansky.
\newblock {\em Commutative rings}.
\newblock Allyn and Bacon Inc., Boston, Mass., 1970.

\bibitem[KR00]{KR}
M. Kontsevich and A.~L. Rosenberg.
\newblock Noncommutative smooth spaces.
\newblock In {\em The Gelfand Mathematical Seminars, 1996--1999}, Gelfand Math.
  Sem., pages 85--108. Birkh\"auser Boston, Boston, MA, 2000.

\bibitem[Lad08]{Ladkani}
S. Ladkani.
\newblock On derived equivalences of categories of sheaves over finite posets.
\newblock {\em J. Pure Appl. Algebra}, 212(2):435--451, 2008.

\bibitem[MR87]{McRob1987}
J.~C. McConnell and J.~C. Robson.
\newblock {\em Noncommutative {N}oetherian rings}.
\newblock Pure and Applied Mathematics (New York). John Wiley \& Sons Ltd.,
  Chichester, 1987.
\newblock With the cooperation of L. W. Small, A Wiley-Interscience
  Publication.

\bibitem[PP79]{popescu}
N.~Popescu and L.~Popescu.
\newblock {\em Theory of categories}.
\newblock Martinus Nijhoff Publishers, The Hague, 1979.

\bibitem[Ros90]{RosenbergLevitzki}
Alexander~L. Rosenberg.
\newblock The left spectrum, the {L}evitzki radical, and noncommutative
  schemes.
\newblock {\em Proc. Nat. Acad. Sci. U.S.A.}, 87(21):8583--8586, 1990.

\bibitem[Ros95]{Rosenbergbook}
A.~L. Rosenberg.
\newblock {\em Noncommutative algebraic geometry and representations of
  quantized algebras}, volume 330 of {\em Mathematics and its Applications}.
\newblock Kluwer Academic Publishers Group, Dordrecht, 1995.

\bibitem[Ros98]{Rosenberg98}
A.~L. Rosenberg.
\newblock Noncommutative schemes.
\newblock {\em Compositio Math.}, 112(1):93--125, 1998.

\bibitem[RV70]{Reis}
C.~M. Reis and T.~M. Viswanathan.
\newblock A compactness property for prime ideals in {N}oetherian rings.
\newblock {\em Proc. Amer. Math. Soc.}, 25:353--356, 1970.

\bibitem[Sch85]{Schofieldbook}
A.~H. Schofield.
\newblock {\em Representation of rings over skew fields}, volume~92 of {\em
  London Mathematical Society Lecture Note Series}.
\newblock Cambridge University Press, Cambridge, 1985.

\bibitem[Sha94]{Shafo}
I.~R. Shafarevich.
\newblock {\em Basic algebraic geometry. 2}.
\newblock Springer-Verlag, Berlin, second edition, 1994.
\newblock Schemes and complex manifolds, Translated from the 1988 Russian
  edition by Miles Reid.

\bibitem[{\v{S}}ko06]{Skoda}
Zoran {\v{S}}koda.
\newblock Noncommutative localization in noncommutative geometry.
\newblock In {\em Non-commutative localization in algebra and topology}, volume
  330 of {\em London Math. Soc. Lecture Note Ser.}, pages 220--313. Cambridge
  Univ. Press, Cambridge, 2006.

\bibitem[Smi71]{WWSmith}
W.~W. Smith.
\newblock A covering condition for prime ideals.
\newblock {\em Proc. Amer. Math. Soc.}, 30:451--452, 1971.

\bibitem[Vak07]{vakil}
R.~Vakil.
\newblock Math 216: {F}oundations of {A}lgebraic {G}eometry (lecture notes).
\newblock 2007.
\newblock \texttt{http://math.stanford.edu/\~{}vakil/216/}.

\bibitem[VO00]{VanObook}
F. Van~Oystaeyen.
\newblock {\em Algebraic geometry for associative algebras}, volume 232 of {\em
  Monographs and Textbooks in Pure and Applied Mathematics}.
\newblock Marcel Dekker Inc., New York, 2000.

\end{thebibliography}

\end{document}